\documentclass[11pt,twoside]{amsart}
\usepackage{float}
\usepackage[utf8]{inputenc}
\usepackage[english]{babel}
\usepackage[T1]{fontenc}
\usepackage{mathtools}
\usepackage{enumitem}
\usepackage{diagbox}
\usepackage{array}

\usepackage{tabularx}
\usepackage{graphicx}
\usepackage{tikz}
\usepackage[all]{xy}
\usepackage{tcolorbox}
\usepackage{amsmath,
    amsfonts,
    amssymb,
    mathrsfs,
    stackrel}
\usepackage{comment}
\usepackage{bookmark}
\usepackage{hyperref}
\usepackage{cleveref}
\hypersetup{colorlinks,
    citecolor=blue,
    linktocpage}

\DeclareFontFamily{U}{wncy}{}
\DeclareFontShape{U}{wncy}{m}{n}{<->wncyr10}{}
\DeclareSymbolFont{mcy}{U}{wncy}{m}{n}
\DeclareMathSymbol{\Sha}{\mathord}{mcy}{"58}

\newcommand\cC{{\mathcal C}}

\newcommand\cE{{\mathcal E}}
\newcommand\cF{{\mathcal F}}

\newcommand\cH{{\mathcal H}}
\newcommand\cI{{\mathcal I}}

\newcommand\cL{{\mathcal L}}
\newcommand\cM{{\mathcal M}}

\newcommand\cO{{\mathcal O}}

\newcommand\cT{{\mathcal T}}

\newcommand\fU{{\mathfrak U}}

\newcommand\bC{{\mathbb C}}
\newcommand\bD{{\mathbb D}}

\newcommand\bG{{\mathbb G}}

\newcommand\bN{{\mathbb N}}

\newcommand\bP{{\mathbb P}}
\newcommand\bQ{{\mathbb Q}}
\newcommand\bR{{\mathbb R}}
\newcommand\bS{{\mathbb S}}
\newcommand\bT{{\mathbb T}}

\newcommand\bV{{\mathbb V}}

\newcommand\bZ{{\mathbb Z}}

\newcommand\hB{{\widehat B}}
\newcommand\hC{{\widehat C}}

\newcommand\hX{{\widehat X}}

\newcommand\wD{{\widetilde D}}

\newcommand\wV{{\widetilde V}}

\newcommand\wX{{\widetilde X}}
\newcommand\wY{{\widetilde Y}}

\newcommand\oD{{\overline D}}

\newcommand\oV{{\overline V}}

\DeclareMathOperator{\SL}{SL}

\newtheorem{theorem}{Theorem}[section]

\newtheorem{corollary}[theorem]{Corollary}
\newtheorem{lemma}[theorem]{Lemma}

\newtheorem{proposition}[theorem]{Proposition}

\theoremstyle{definition}

\newtheorem{definition}[theorem]{Definition}
\newtheorem{example}[theorem]{Example}

\newtheorem{notation}[theorem]{Notation}

\newtheorem{observation}[theorem]{Observation}

\newtheorem{remark}[theorem]{Remark}
\newtheorem{exercise}[theorem]{Exercise}

\newcommand{\twopartdef}[4]
{
	\left\{
		\begin{array}{ll}
			#1 & \mbox{if } #2 \\
			#3 & \mbox{if } #4
		\end{array}
	\right.
}

\newcommand{\threepartdef}[6]
{
	\left\{
		\begin{array}{lll}
			#1 & \mbox{if } #2 \\
			#3 & \mbox{if } #4 \\
			#5 & \mbox{if } #6
		\end{array}
	\right.
}

\newcommand{\threepartdefwhen}[6]
{
	\left\{
		\begin{array}{lll}
			#1 & \mbox{when } #2 \\
			#3 & \mbox{when } #4 \\
			#5 & \mbox{when } #6
		\end{array}
	\right.
}

\title[Polync varieties]{Polync varieties 
and multiparameter Kulikov models}
\author[Engel]{Philip Engel}
\address{Department of Mathematics, Statistics, 
and Computer Science, University of Illinois in Chicago (UIC),
851 S Morgan St, Chicago, IL 60607, USA}
\email{pengel@uic.edu}

\begin{document}

\begin{abstract} 
We study 
``polync varieties'', whose singularities
are locally products of normal
crossing (nc) singularities.
We introduce the notion of $d$-semistability
of such varieties, 
and generalize work of Friedman and
Kawamata--Namikawa to address the
smoothability of $d$-semistable,
$K$-trivial, polync
varieties. These results are applications
of recent breakthroughs on the logarithmic
Bogomolov--Tian--Todorov theorem,
due to Chan--Leung--Ma
and Felten--Filip--Ruddat.
We generalize the combinatorial description of Kulikov models 
for K3 surfaces to the setting of a multiparameter base 
and describe some interesting examples.
\end{abstract}

\setcounter{tocdepth}{2}
\maketitle 

\tableofcontents

\section{Polync varieties}

\begin{definition} A {\it polync variety},
or {\it polync analytic space}, is an algebraic variety, resp.~analytic space, 
that is \'etale, resp.~analytically, locally
isomorphic to a product
\begin{align}\label{product}
\prod_{i\in I} \,\{x_i^{(1)}\cdots x_i^{(m_i)}=0\}
\end{align}
 of
normal crossing (nc) singularities, 
$m_i\geq 2$, and a smooth
space $\bC^m$. We
say that a polync variety/analytic space $X_0$
is {\it polysnc} if the irreducible
components of $X_0$ are smooth.
\end{definition}

Such singularities
have also been
called ``generalized semi-stable'',
see \cite[Sec.~0]{li}.


\begin{example}
Any nc variety $X_0$ is polync; we
need only $|I|=1$ factor in the local
equation (\ref{product}), 
at any point $x\in X_0$.
Any product of nc varieties is polync.
\end{example}

\begin{definition}
A {\it polysimplex} is a polyhedron,
which is a product of simplices
$P_I\coloneqq \prod_{i\in I} \sigma_{m_i-1}$
for $m_i\geq 2$ (for $P_I$ a point,
we take the empty product). 
It has real dimension 
$\sum_{i\in I} (m_i-1)$.
\end{definition}

\begin{example}
Any simplex is a polysimplex with only one factor. 
A cube in any dimension is a polysimplex, 
being a product of intervals, i.e.~$1$-simplices.
A triangular prism is a polysimplex,
as the product of a $1$-simplex and a 
$2$-simplex. See Figure \ref{polysim}.
\end{example}

Let $X_0$ be polync.
A {\it closed stratum} 
$X_\sigma\subset X_0$ 
is
an irreducible component of 
an intersection of 
components of $X_0$.
The corresponding {\it open stratum} 
$(X_\sigma)^\circ\subset X_\sigma$ is
the complement
of all lower-dimensional closed strata
$X_\tau\subsetneq X_\sigma$. 
Along
$(X_\sigma)^\circ$, we
have a local product
form (\ref{product}).
We denote by $I=I(\sigma)$
the indexing set for
the local nc factors
appearing in this product.

\begin{figure}
    \centering
    \includegraphics[width=2.7in]{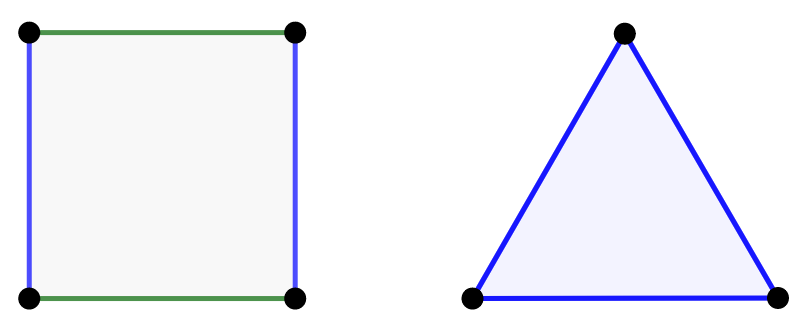}
    \hspace{20pt}
    \includegraphics[width=1.1in]{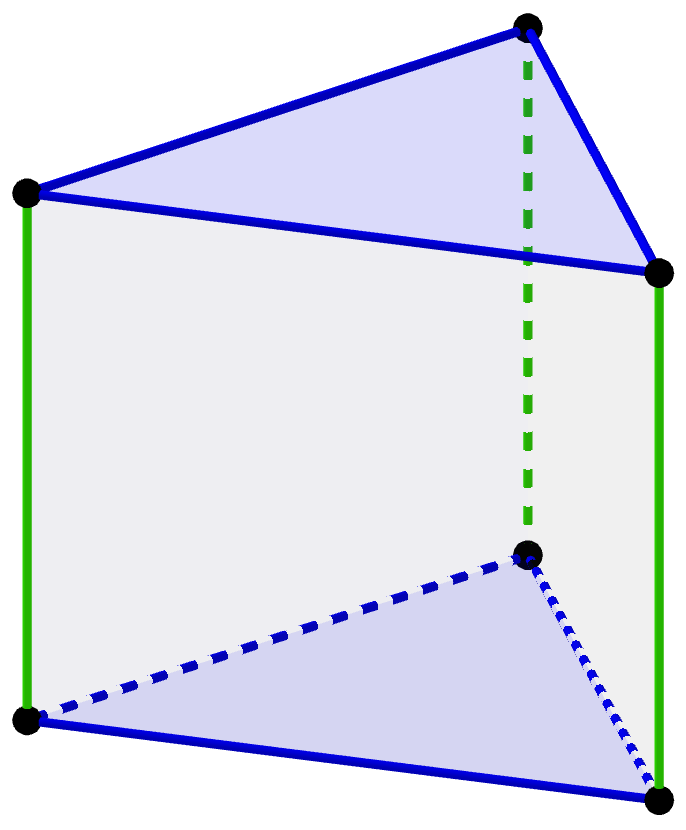}
    \caption{Two $2$-dimensional polysimplices and one $3$-dimensional
    polysimplex}
    \label{polysim}
\end{figure}

\begin{figure}
    \centering
    \includegraphics[width=3.8in]{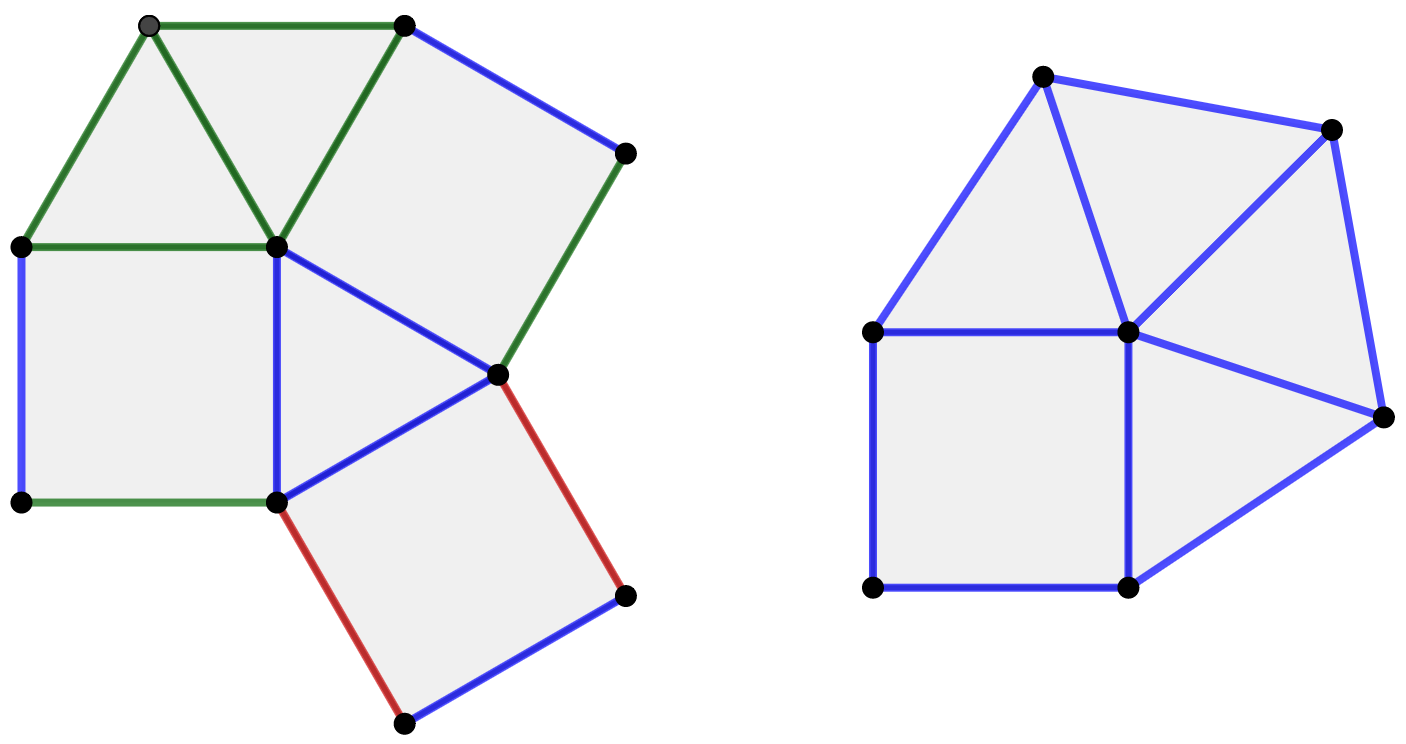}
    \caption{Left:~Colorable polysimplicial
    complex of dimension $2$, colored by the 
    set $S = \{{\rm blue},\,{\rm green},\,{\rm red}\}$. Right:~$2$-dimensional polysimplicial complex
    that admits no coloring; the chain
    of $2$-simplices forces the two $1$-simplex factors of the square to be colored
    by the same element of $S$.
    }
    \label{polysim2}
\end{figure}

\begin{definition} The {\it dual complex}
$\Gamma(X_0)$
of a polync variety is a polyhedral complex 
defined as follows: For each open
singular stratum $(X_\sigma)^\circ\subset X_0$ along 
which $X_0$ has the local form 
given by
(\ref{product}), we take a polysimplex
$P_\sigma \coloneqq \prod_{i\in I(\sigma)} \sigma_{m_i-1}.$
If we have an inclusion
$X_\tau\subset X_\sigma$ then $P_\tau$ naturally contains
$P_\sigma$ as a face. Then,
we glue along the poset
of singular strata $\{X_\sigma\}$
of $X_0$ to get a polyhedral complex
$$\textstyle \Gamma(X_0)=\bigcup_\sigma P_\sigma.$$
\end{definition}

\begin{definition}\label{colorable}
A polyhedral complex 
$\Gamma = \bigcup_\sigma P_\sigma$
whose cells are polysimplices,
is a {\it polysimplicial complex}. 
A polysimplicial complex 
is {\it colorable} if there exists
a finite set $S$, together with
a collection of injective
maps $\phi_\sigma\,\colon I(\sigma) \to S$, one for 
each polysimplex $P_\sigma\subset \Gamma$,
such that the functions $\phi_\sigma$
are compatible along inclusions
of faces. That is, $\phi_\tau\vert_{I(\sigma)}=
\phi_\sigma$  for any face $P_\sigma\subset P_\tau$. 
We call $\{\phi_\sigma\}$ 
an {\it $S$-coloring}
of  $\Gamma$.
\end{definition}

That is, there is a coloring
of each positive-dimensional 
simplex factor of 
$P_\sigma$ by \textbf{distinct} 
elements of $S$. Furthermore,
these colorings
are compatible with inclusions 
$P_\sigma\subset P_\tau$ of polysimplices.
See Figures \ref{polysim2} and \ref{fig:rhomb}.
In particular, the $1$-skeleton
$\Gamma^{[1]}\subset \Gamma$ is an
$S$-colored graph, as any $1$-dimensional
polysimplex is a $1$-simplex, and so 
has only one simplex factor; $I=\{i\}$.

\begin{notation} We write 
$\phi_{\sigma}(i) = s(i)\in S$ for the color
of the $i$th simplex
factor of $P_\sigma$.\end{notation}

\begin{example}
Any simplicial complex
is colorable by a single
color $S = \{s\}$.
It is also sometimes possible 
to use multiple colors. 
For example, a $1$-dimensional 
simplicial complex is a graph $G$. One
may take $S = E(G)$ to be the edges, 
assigning each edge a different color.
\end{example}

\begin{definition} 
We say that a polync variety
$X_0$ is {\it colorable} if
its dual polysimplicial complex
$\Gamma(X_0)$ is colorable;
an {\it $S$-coloring} of $X_0$ 
is an $S$-coloring of $\Gamma(X_0)$.
\end{definition}


\begin{definition}
Suppose $X_0$ is an $S$-colored, polync variety.
The {\it color $s\in S$ singular locus} 
$(X_0)_{\rm sing}^s\subset (X_0)_{\rm sing}$
is the union of all strata $X_\sigma$ for which
$s$ appears as a color of a simplex
factor of $P_\sigma\subset \Gamma(X_0)$, 
or equivalently, $s\in {\rm im}(\phi_{\sigma})$.
\end{definition}

\begin{remark}
$(X_0)_{\rm sing}^s\subset X_0$ 
is a divisor
because 
every stratum $X_\sigma$ involving
$s\in S$ is contained in a codimension
$1$ stratum of color $s$, 
associated to any edge
of the polysimplex $P_\sigma$
 of color $s$.
We have a decomposition
$(X_0)_{\rm sing} = \bigcup_{s\in S}
(X_0)_{\rm sing}^s$ and
the intersection
$(X_0)_{\rm sing}^s\cap (X_0)_{\rm sing}^{s'}$ 
for $s\neq s'\in S$
contains no divisors,
corresponding to the fact 
that the graph $\Gamma(X_0)^{[1]}$ is $S$-colored.
\end{remark}

See Figure \ref{fig:rhomb-dual} for a 
polync surface and the color components
of its singular locus.

\begin{proposition}\label{prop:1st}
Let $X_0$ be an $S$-colored, 
polync variety.
The sheaf 
$\cE xt^1(\Omega^1_{X_0},\cO_{X_0})\simeq \textstyle \bigoplus_{s\in S} \cL_s$
is a direct sum of line bundles 
$\cL_s$ supported on 
$(X_0)_{\rm sing}^s$.
\end{proposition}

\begin{proof}
We first show the analogous local statement: 
\begin{align}\label{loc-prod}
\cE xt^1 (\Omega^1_U, \cO_U)\simeq 
\textstyle\bigoplus_{i\in I} 
\cO_{(U_i)_{\rm sing}}\end{align}
where $U \simeq \prod_{i\in I} U_i\times \bC^m$ 
and $U_i \simeq\{x_i^{(1)}\cdots 
x_i^{(m_i)}=0\}\subset \bC^{m_i}$ 
is the local product form
(\ref{product}).
The isomorphism (\ref{loc-prod})
follows from the 
K\"unneth formula:
\begin{align}\label{ext}\textstyle 
\cE xt^1(\Omega^1_U,\cO_U) \simeq 
\bigoplus_{i\in I}
\cE xt^1(\Omega^1_{U_i}, \cO_{U_i}) \simeq
\bigoplus_{i\in I} \cO_{(U_i)_{\rm sing}}.\end{align} 
The second isomorphism 
in (\ref{ext}) is the well-known
statement that the first-order deformations
of nc singularities are $1$-dimensional.
The local
statement globalizes to the desired global
statement by the colorability, which ensures
that a summand $\cO_{(U_i)_{\rm sing}}$ in (\ref{ext})
of a given
color $s(i)=s\in S$ is 
identified in an overlapping local coordinate
patch with a factor of the same color $s$.
\end{proof}

\begin{definition}\label{defn:dss}
An $S$-colored, polync variety
is {\it $d$-semistable} if 
$\cL_s\simeq \cO_{(X_0)^s_{\rm sing}}$
for all $s\in S$.
\end{definition}

As we will see, this generalizes
Friedman's original notion of 
$d$-semistability
\cite[Def.~1.13]{friedman}
for normal crossing degenerations over
a $1$-parameter base, 
to higher-dimensional bases.

\begin{definition}
An $S$-colored, polync variety $X_0$
is {\it smoothable} if there is a flat,
proper family $$\textstyle
X\to {\rm Spec} \,\,\bC[[u_s]]_{s\in S}$$ 
such that the central fiber
is $X_0$ and there is a local form (\ref{product}) on $X_0$
extending to a local form 
\begin{align}\label{product2}\textstyle 
\prod_{i\in I} \{x_i^{(1)}\cdots x_i^{(m_i)}=u_{s(i)}\}
\end{align}
on $X$. Here $s(i)$ is the color of the
$i$th factor. Similarly, we say
$X_0$ is {\it formally smoothable}
if we have a formal deformation
$\hX\to {\rm Spf} \,\,\bC[[u_s]]_{s\in S}$
with the same local form.
\end{definition}

\begin{definition}\label{semistable}
    We call
a morphism $X\to Y$ over a smooth base,
which is \'etale- or analytically-locally 
of the form (\ref{product2}) a {\it semistable morphism},
following \cite[Sec.~0.3]{ak},
\cite[Sec.~1.1]{alt}.
If additionally, $\Gamma(X_0)$ is 
a cube complex (i.e.~every polysimplex
is a product of intervals), we say
that $X\to Y$ is a {\it nodal morphism},
see \cite[Def.~5.1]{survey}, 
\cite[Def.~2.4]{egfs}.
\end{definition}

In particular, the total space
$X$ of a semistable morphism,
or a smoothing, is regular. 

\begin{proposition} A
(formally) smoothable, $S$-colored,
polync variety is $d$-semistable.
\end{proposition}

\begin{proof}
Consider the Kodaira--Spencer map 
\begin{align*} {\rm ks}\colon T_0\, {\rm Spec} \,\,\bC[[u_s]]_{s\in S} \to {\rm Ext}^1(\Omega^1_{X_0},\cO_{X_0}).
\end{align*}
We have a map
$\phi\colon
{\rm Ext}^1(\Omega^1_{X_0},\cO_{X_0})\to
H^0(X_0, \,\cE xt^1(\Omega^1_{X_0},\cO_{X_0}))$
by the local-to-global spectral
sequence. It
sends a global first-order
deformation to a collection
of compatible local first-order
deformations.
By the required
local form (\ref{product2}) of a 
smoothing, 
$(\phi\circ {\rm ks})(\tfrac{\partial}{\partial u_s})$
represents a nowhere vanishing local
section of $\cL_s$ along $(X_0)^s_{\rm sing}$
with respect to the local product expansion 
(\ref{ext}). It follows that 
$\cL_s\simeq \cO_{(X_0)^s_{\rm sing}}$
for all $s\in S$.
\end{proof}

\begin{remark}
``Polysnc'' is the smallest
class of singularities on fibers
of a flat morphism,
achievable by alteration. 
More precisely, let $f\colon X\to Y$
be any morphism of varieties.
By the resolution 
of the Abramovich--Karu conjecture \cite[Thm.~0.3]{ak},
\cite[Thm.~4.5]{alt}, there is an 
alteration $Y'\to Y$ and a birational
modification of the base change $f'\colon X'\to Y'$ for which $f'$ is (flat and) semistable.
And by Hodge-/weight-theoretic
considerations, it is impossible
to achieve a ``better'' result.
In this sense, polysnc singularities
and semistable morphisms, over
bases of arbitrary dimension, play exactly
the same role as snc singularities
and semistable morphisms over 
$1$-parameter bases.
\end{remark}

\section{Interpretation via log structures}

Following French
conventions, let $\bN\coloneqq \bZ_{\geq 0}$.
Let $0^S:= {\rm Spec}(\bN^S\to \bC)$ 
denote the ``logarithmic point''
with ghost sheaf $\bN^S$. It is the log
structure on ${\rm Spec}\,\bC$
with monoid $\bC^*\oplus \bN^S$ 
and structural morphism 
$(c, \sum_{s\in S} n_s e_s)\mapsto c\cdot 
0^{\sum_{s\in S} n_s}$, where by 
convention $0^0=1$. It is
the pullback of the divisorial log
structure on ${\rm Spec}\,\bC[[u_s]]_{s\in S}$ associated to the union 
of the coordinate hyperplanes,
to the origin $0\in {\rm Spec}\,\bC[[u_s]]_{s\in S}$.
See \cite{log}
for background on logarithmic geometry.

\begin{definition}\label{dss}
Let $X_0$ be an $S$-colored,
polync variety.
A {\it log structure of semistable type} 
on $X_0$
is the following data:
\begin{enumerate}
    \item a log structure $\cM_0$ on $X_0$ 
admitting \'etale charts \begin{align*} \left(\textstyle 
\bigoplus_{i\in I} \bN^{m_i}\right) \oplus \bN^{S\setminus \phi(I)} &\to \cO_{X_0,\, \textrm{\'et}}(U) \\ 
e_i^{(j)}&\mapsto x_i^{(j)} \textrm{ for }i\in I,\, j\in \{1,\dots,m_i\}, \\
e_s &\mapsto 0 \textrm{ for }s\in S\setminus \phi(I),
\end{align*}
where $x_i^{(j)}$ define 
\'etale-local coordinate charts
on $U$ to the local form (\ref{product}), and $\phi$
is the local $S$-coloring of the
factors $i\in I$,
\item
a morphism $f\colon (X_0,\cM_0)\to 
({\rm Spec}\,\bC, 0^S)$ 
of log schemes, 
which is given in local charts by
the monoid morphism $f^\flat\colon \bN^S\to \left(\bigoplus_{i\in I}\bN^{m_i}\right) \oplus \bN^{S\setminus \phi(I)}$ sending
\begin{align*} \textstyle f^\flat e_s &= \textstyle \sum_{j=1}^{m_i} e_i^{(j)}\textrm{ if }s=\phi(i), \\
f^\flat e_s &= e_s\textrm{ for }s\in S\setminus\phi(I).
\end{align*}
\end{enumerate}
\end{definition}

We will henceforth drop the $\textrm{\'et}$
subscript in the structure sheaf.
Definition \ref{dss} is very closely
related to a similar notion defined
in \cite[Def.~5.5]{li}, which does
not require $S$-colorability, but
in a sense, only demands 
the existence of $f$
``along the diagonal''
$({\rm Spec}\,\bC,0)\hookrightarrow ({\rm Spec}\,\bC,0^S)$.

\begin{proposition}\label{equiv-dss}
Let $X_0$ be an $S$-colored, polync variety.
Then $X_0$ is $d$-semistable if and only if 
it admits a log structure of semistable type,
as in Definition \ref{dss}.
\end{proposition}

\begin{proof} When $X_0$ is a normal
crossing variety
and $S=\{s\}$ is a singleton, this is 
\cite[Prop.~1.1]{nn}, see also
\cite[Thm.~11.7]{kato} and its arXiv
version \cite[Thm.~11.2]{kato_arxiv}. 
The same proof as \cite{kato_arxiv}
applies nearly verbatim, but we outline 
the necessary changes here
for completeness.
To prove the forward direction,
suppose $X_0$ is $d$-semistable, i.e.~
\begin{align}\label{iso}\textstyle
\cE xt^1(\Omega^1_{X_0},\cO_{X_0})\simeq 
\bigoplus_{s\in S} \cO_{(X_0)_{\rm sing}^s}.\end{align} 
Then,
we may choose elements of 
${\bf 1}_s\in H^0(X_0, \,\cE xt^1(\Omega^1_{X_0},\cO_{X_0}))$ 
where via an isomorphism (\ref{iso}),
${\bf 1}_s=(f_{s'})_{s'\in S}$ 
with $f_s=1$ and $f_{s'}=0$
for $s'\neq s$.
Taking the (K\"unneth) product
of the formula given in 
\cite[Lem.~12.1]{kato_arxiv},
we have, in a local coordinate 
chart (\ref{product}), an isomorphism
$$\textstyle
\cE xt^1(\Omega^1_{X_0},\cO_{X_0})\simeq \bigoplus_{s\in S} (\cI_{X_0}/
\cI_{X_0}\cI_{(X_0)_{\rm sing}^s})^\vee\otimes \cO_{(X_0)_{\rm sing}^s}$$
where $\cI_{X_0} = (x_i^{(1)}\cdots x_i^{(m_i)})_{i\in I}$ is the ideal of 
$X_0\subset \prod_{i\in I}\bC^{m_i} \times \bC^m$.
Following
\cite[Prop.~12.3]{kato_arxiv}, 
we have a projection 
$p_s\colon \cI_{X_0}/\cI_{X_0}^2\to 
\cI_{X_0}/\cI_{X_0}\cI_{(X_0)_{\rm sing}^s}$
and the section of the latter
sheaf corresponding to ${\bf 1}_s$ is, locally
over an open cover $\{U_\lambda^s\}$ of 
$(X_0)_{\rm sing}^s$,
the projection of a section $$p_s(v_s x_i^{(1)}\cdots x_i^{(m_i)})$$ where $s(i)=s$
and $v_s\in \cO_{X_0}^*(U_\lambda^s)$.
Set $$\zeta_i^{(j),\,\lambda}\coloneqq \threepartdef{v_sx_i^{(1)}}{j=1,}{x_i^{(j)}}{j=2,\dots,m_i,}{1}{j=m_i+1,\dots, \dim X_0+1.}$$
Then it follows exactly as in \cite[Prop.~12.2]{kato_arxiv} that there is a ``transition system''
of units $u_s^{(j),\,{\lambda\mu}}\in \cO_{X_0}^*(U_\lambda^s\cap U_\mu^s)$ and a permutation
$\sigma\in S_{m_i}$ such that 
$$\zeta_i^{(\sigma\cdot j),\,\lambda} = \zeta_{i'}^{(j),\,\mu}u_s^{(j),\,\lambda\mu}.$$
Here $i'\in I'$ is the index for
the coordinate
chart $U_\mu^s$ satisfying $s(i')=s$.
Furthermore, it follows that $\prod_j u_s^{(j),\,\lambda\mu}\vert_{(X_0)_{\rm sing}^s}=1$
by the condition that the $\zeta_i^{(j),\,\lambda}$
are produced from a global section
of $\cI_{X_0}/\cI_{X_0}\cI_{(X_0)_{\rm sing}^s}$.
By the same formula as in 
\cite[Proof of Prop.~12.3]{kato_arxiv},
it is possible to modify the transition
system $$\{u_s^{(j),\,\lambda\mu}\}_{j, \lambda,\mu}$$
(for each $s\in S$) to a new
transition system which satisfies the stronger
property
$\prod_j u_s^{(j),\,\lambda\mu}=1$
on all of $U^s_\lambda\cap U_\mu^s$.

Take a common
refinement of the collections
$\{U_\lambda^s\}$ for all $s\in S$
and adjoin \'etale coordinate charts
covering $X_0\setminus (X_0)_{\rm sing}$
to produce an 
\'etale open cover $\fU\coloneqq \{U_\lambda\}$
of $X_0$.
Then, for every
$s\in S$, the transition system 
of color $s$ defined in the 
previous paragraph restricts to this 
common cover and satisfies 
$\prod_j u_s^{(j),\lambda\mu}=1$
on every double overlap of the
color $s$ subcover $\fU_s$.

Using \cite[Eqn.~(12)]{kato_arxiv}, we define
charts of a log structure 
$\zeta\colon  
\left( \bigoplus_{i\in I}\bN^{m_i}\right) \oplus \bN^{S\setminus \phi(I)}\to \cO_{X_0}(U_\lambda)$
on $U_\lambda$ 
by sending $e_i^{(j)}\mapsto 
\zeta_i^{(j),\,\lambda}$ for $i\in I$,
and $e_s\mapsto 0$ for $s\in S\setminus \phi(I)$. Call it $\cM_\lambda$
and let $\cM_\lambda^s\subset \cM_\lambda$ be the sub-log structure corresponding to the color $s\in S$. 
Note that this definition 
is valid on open sets
$U_\lambda$ involving singular strata
of multiple colors;
the local coordinates
$\zeta_i^{(j),\,\lambda}$
for each color
are chosen independently.

\begin{enumerate}
\item\label{c1} When $U_\lambda$, $U_\mu\in \fU_s$,
the transition
system $\{u_s^{(j),\,\lambda\mu}\}$,
together with the corresponding
coordinate permutations defines
an isomorphism $\alpha_{\lambda\mu}^s\colon 
\cM_\lambda^s\vert_{U_{\lambda\mu}}\to 
\cM_\mu^s\vert_{U_{\lambda\mu}}$ 
between the color $s$
part of the log structure on a
double overlap
of $U_\lambda$, $U_\mu\in \fU_s$, as in
\cite[Eqn.~(13)]{kato_arxiv}.
Since
$\prod_j u_s^{(j),\,\lambda\mu}=1$,
it identifies $\sum_j e_i^{(j)}$ and
$\sum_j e_{i'}^{(j)}$, where
$s(i)=s(i')=s$. \smallskip

\item\label{c2} When 
$U_\lambda\in \fU_s$ 
and $U_\mu\notin \fU_s$ (or vice
versa), we uniquely specify an isomorphism
$\alpha_{\lambda\mu}^s\colon
\cM_\lambda^s\vert_{U_{\lambda\mu}}\to 
\cM_\mu^s\vert_{U_{\lambda\mu}}$
by requiring that it act
trivially on units and
identify the elements
$\sum_j e_i^{(j)}\leftrightarrow e_s$ 
where $s(i)=s.$ 

\item\label{c3}
When $U_\lambda$, $U_\mu\notin \fU_s$ we demand
that the isomorphism $\alpha_{\lambda\mu}^s$
respect the element $e_s$ and act
trivially on units. 
\end{enumerate}

Then, the above
isomorphisms $\alpha_{\lambda\mu}^s$
defined in (\ref{c1}, \ref{c2}, \ref{c3})
for all $s\in S$ collectively
determine a unique isomorphism 
$\alpha_{\lambda\mu}\colon
\cM_\lambda\vert_{U_{\lambda\mu}}
\to\cM_\mu\vert_{U_{\lambda\mu}}$.

For each $s\in S$, the cocycle
condition holds for the isomorphisms
$\alpha_{\lambda\mu}^s$
on triple overlaps
of elements of $\fU_s$ by
\cite[p.~28, Step 1]{kato_arxiv}.
Every other triple overlap is disjoint
from $(X_0)_{\rm sing}^s$ so here
the characteristic monoid in color
$s$ is $\bN$, and
the cocycle condition follows
from the uniqueness of the
isomorphisms fixing units and
identifying the distinguished
generators of the log charts. 
So the isomorphisms 
$\alpha_{\lambda\mu}$ satisfy the cocycle
condition and the local log structures 
$\cM_\lambda$
glue to give a global log structure
$(X_0,\cM_0)$.

Finally, as in
\cite[p.~29, Step 2]{kato_arxiv},
define on each $U_\lambda$ a homomorphism
\begin{align*}
f_\lambda^\flat\colon
\bC^*\oplus\bN^S
&\to
\cO_{X_0}^*(U_\lambda)
\oplus_{\zeta}\textstyle
\left(\left(\bigoplus_{i\in I}
\bN^{m_i}\right)
\oplus\bN^{S\setminus\phi(I)}\right),\\
f_\lambda^\flat(1,e_s)
&\coloneqq
\twopartdef{
(1,\sum_{j=1}^{m_i}e_i^{(j)})}
{s=\phi(i),}{
(1,e_s)}
{s\notin\phi(I).}
\end{align*}
By construction,
$\alpha_{\lambda\mu}$ identifies
$f_\lambda^\flat(1,e_s)$ with
$f_\mu^\flat(1,e_s)$ for every
$s\in S$. In case (\ref{c1}), this
follows from
$\prod_j u_s^{(j),\,\lambda\mu}=1$,
and in cases (\ref{c2}) and
(\ref{c3}) it holds by definition.
Consequently, the local homomorphisms
$f_\lambda^\flat$ glue to give a
global morphism of log schemes
\[
f\colon (X_0,\cM_0)\to
({\rm Spec}\,\bC,0^S).
\]

This completes the proof of the forward direction.
We omit the reverse direction, which follows 
in the same manner as \cite{kato_arxiv}, taking
into account the above modifications.
\end{proof}

\begin{proposition}\label{logsmooth}
Suppose that $X_0$ is a
$d$-semistable, $S$-colored,
polync variety.
The corresponding
log structure of semistable type $(X_0,\cM_0)\to 
({\rm Spec}\,\bC, 0^S)$ is log smooth
and saturated.
\end{proposition}

\begin{proof}
This holds by 
\cite[Thm.~3.5]{kato_old} and applying
the definition of saturated morphisms. 
\end{proof}

It follows from \cite[4.5]{kato_old}
that the logarithmic deformations
of $(X_0,\cM_0)\to B_0 = ({\rm Spec}\,\bC, 0^S)$ are flat in the 
usual sense. Furthermore,
as a log smooth morphism, 
the log cotangent
bundle $\Omega_{X_0/B_0}^{\rm log}$ is a vector bundle, and the 
log canonical bundle 
$\omega_{X_0/B_0}^{\rm log}=
\wedge^{\dim X_0}\Omega_{X_0/B_0}^{\rm log}$ is a line bundle. Then, we say that:

\begin{definition} $X_0$ is 
{\it log Calabi--Yau} (over $B_0$) if 
$\omega_{X_0/B_0}^{\rm log}\simeq \cO_{X_0}$.
\end{definition}

In general, for log structures,
this does not mean that the usual dualizing
sheaf $\omega_{X_0}$ is trivial. For instance, 
$X_0$ could be a smooth projective variety
with a nontrivial snc divisor $D_0\in |-K_{X_0}|$ and $\cM_0$ could be the 
divisorial log structure 
associated to $D_0$.
But in the case of a log structure of semistable type, we do have $\omega_{X_0/B_0}^{\rm log}=\omega_{X_0}$.

Sections
of the logarithmic
cotangent, resp.~canonical, bundle
$\Omega_{X_0/B_0}^{\rm log}$,
resp.~$\omega_{X_0/B_0}^{\rm log}$,
are logarithmic one-forms,
resp.~top-forms,
on the components, whose residues
cancel on double loci.
These facts follow easily from the nc
case, by taking direct
sums, resp.~tensor products,
of the relevant
sheaves on nc factors of (\ref{product}).

Finally, we observe that $\bN^S$
is a sharp toric monoid---it is finitely
generated, integral, saturated, 
torsion-free, and (the sharpness
condition) $0\in \bN^S$ is the only 
invertible element.
Combined with Proposition \ref{logsmooth},
we may apply the main result
\cite[Thm.~1.1]{fp} (similar results
were obtained earlier in 
\cite{btt2, btt3}; we found
the presentation
\cite{fp} quite clear) to deduce:

\begin{theorem}\label{smoothable}
Suppose that $X_0$ is a proper,
$d$-semistable,
$S$-colored, polync variety
with trivial dualizing sheaf $\omega_{X_0}\simeq \cO_{X_0}$
and all algebraic components. Let $\cM_0$
be a log structure on $X_0$ of $d$-semistable
type. 
The log smooth deformations
of $(X_0,\cM_0)/B_0=({\rm Spec}\,\bC, 0^S)$ are 
unobstructed.
\end{theorem}

\begin{corollary}\label{universal}
There is a universal formal
log smooth deformation
$$
\hX\to \hB\coloneqq {\rm Def}_{X_0/B_0}\simeq 
{\rm Spf}\,\bC[[u_s,v_t]]_{s\in S,\,t\in T}
$$
where $\{0\}\times
{\rm Spf}\,\bC[[v_t]]_{t\in T}$
is the locus of locally trivial, log smooth
deformations of 
$X_0/B_0$.
\end{corollary}

\begin{proof}
By \cite[Thm.~8.7]{kato}, the logarithmic
deformation functor admits a miniversal
formal deformation $\pi\colon \hX\to\hB$.
By Theorem \ref{smoothable}, $\hB$
is formally smooth, and hence has the
above form.

Let $\eta_0$ be a generator of
$\omega_{X_0/B_0}^{\rm log}\simeq \cO_{X_0}$.
It lifts to a (nonvanishing)
generator $\eta$ of
$\omega_{\hX/\hB}^{\rm log}$ by
the local freeness of the (logarithmic) Hodge 
bundle $$\cH^{d,0}\coloneqq \pi_* \omega_{\hX/\hB}^{\rm log}$$
for $d=\dim X_0$,
see \cite[Thm.~7.1]{Hodge-const}.
Contraction with $\eta$ gives an isomorphism
\[
\Theta_{\hX/\hB}^{\rm log}\simeq
\wedge^{d-1}\Omega_{\hX/\hB}^{{\rm log}}
\]
where $\Theta^{\rm log}_{\hX/\hB}$
is the relative logarithmic tangent bundle.
Again by \cite[Thm.~7.1]{Hodge-const},
$$\cH^{d-1,0}\coloneqq \pi_*(\wedge^{d-1}\Omega_{\hX/\hB}^{{\rm log}})$$
is a vector bundle, so
$H^0(X_b,\Theta_{X_b}^{\rm log})$
has constant dimension for $b\in\hB$.
Since $\hB$ is smooth, and in particular
reduced, Wavrik's criterion
\cite[Thm.~4.2]{wavrik} shows that the
miniversal deformation is universal 
(the proof of the criterion applies
in the log smooth setting,
with vertical vector fields replaced
by logarithmic vertical vector fields).
\end{proof}

Alternatively, 
there is
an analytic smoothing 
$X\to B= \Delta^S\times \Delta^T$
over a polydisk, which agrees 
with the family $\hX$ up to any
specified order
\cite[Thm.~B.1]{sr}. In particular,
once $X$ and $\hX$ agree up to order $2$,
$X$ is a smooth analytic
space, i.e.~a complex manifold,
and the morphism $X\to B$
is an analytic, semistable morphism 
(see Def.~\ref{semistable}).
Its logarithmic Kodaira--Spencer map is an 
isomorphism, so $X\to B$ is miniversal. 
The formal completion is therefore 
isomorphic to $\hX\to\hB$, 
by Corollary \ref{universal},
and so $X\to B$
is universal as well.

If additionally, $X_0$
is a projective variety with ample
line bundle
$L_0$ then by \cite[Ch.~11]{felten},
the deformations of the pair $(X_0, L_0)$
are also unobstructed. In turn, by Grothendieck
algebraization,
we can produce
not just a formal smoothing, but
a projective algebraic smoothing $(X,L)\to {\rm Spec}\,\bC[[u_s,v_t]]_{s\in S,\,t\in T'}$
where $|T'|\leq |T|$ could be smaller,
because we only take deformations 
keeping the line bundle.

\begin{remark}
The original version of Theorem \ref{smoothable}
is Friedman's result \cite[Thm.~5.10]{friedman}
on the smoothability of $d$-semistable,
normal crossings K3 surfaces 
$X_0$ (see Def.~\ref{kulikov-k3}
below). The ``disadvantage'' of a classical
deformation theory approach is that
${\rm Def}_{X_0} = \Delta^{20}\cup_{\Delta^{19}}\Delta^N $
is a union of two components. One
is the smoothing component
$\Delta^{20}$ over which the general
fiber is an analytic K3 surface. The
other ``undesirable''  component
$\Delta^N$ is a large
component of locally trivial nc deformations,
whose general fiber is not $d$-semistable.
These components
meet along a divisor $\Delta^{19}\subset \Delta^{20}$ of locally trivial, $d$-semistable
deformations.

On the other hand, the
``log smooth deformations'' 
${\rm Def}_{X_0/B_0}$ map into
the smoothing
component $\Delta^{20}$, as the
fibers over $\Delta^N\setminus \Delta^{19}$
are not $d$-semistable, and so
do not admit a log structure (of semistable type). Though, caveat 
lector: In the Type III
snc case, $\dim {\rm Def}_{X_0/B_0}=21$ and
to recover a $20$-dimensional base,
we would have to rigidify the deformation
functor ${\rm Def}_{X_0/B_0}$.
See Propositions \ref{reparam-action}, 
\ref{reparam-quotient}, and \ref{bigger}.

The first generalization of Friedman's work 
to higher-dimensional $K$-trivial varieties
was the work of Kawamata--Namikawa \cite{nn},
who reinterpreted it in terms
of logarithmic geometry, and extended
it to 
$K$-trivial, $d$-semistable, nc varieties,
that satisfy the additional hypotheses
$H^{\dim X_0-1}(X_0,\cO_{X_0})=0$
and $H^{\dim X_0-2}(X_0^\nu,\cO_{X_0^\nu})=0$
for $X_0^\nu =\bigsqcup_i X_i$ the normalization.

The key insight leading to a more general statement
(a Bogomolov--Tian--Todorov type 
unobstructedness theorem for the log smooth
deformations of a general 
log smooth, log Calabi--Yau variety, with
no additional technical hypotheses)
came in the work of Chan--Leung--Ma \cite{btt2}, who introduced
the correct dGLA-theoretic set-up, of 
``$\Lambda$-curved 
Batalin–Vilkovisky algebras''. 
We recommend
Felten
\cite{felten} for an 
excellent exposition on
the topic and its history. 

Until recent work
\cite{btt2, btt3, fp}, the 
smoothability question has largely
focused on $1$-parameter degenerations.
But as \cite{survey, egfs} show, multiparameter
degenerations can, by virtue
of their additional combinatorial information,
be powerful tools to address
questions about algebraic cycles, 
(stable) rationality, etc.  In particular,
multiparameter semistable degenerations,
such as the universal deformation of 
a $d$-semistable, $S$-colored, polync variety
have the remarkable property that their
total space $X$ is smooth, allowing one to
specialize algebraic cycle classes transverse to the central fiber, see
e.g.~\cite[Lem.~2.16, Lem.~4.14]{egfs}.
\end{remark}

\begin{definition}\label{kulikov-k3}
A {\it Kulikov model} $f\colon X\to Y$
is a relatively $K$-trivial, semistable morphism.
\end{definition}

As noted, Theorem \ref{smoothable} implies
that the universal log smooth deformation
$X\to B$ 
of a $d$-semistable,
$S$-colored, polync, $K$-trivial variety
is a Kulikov model. Furthermore, the restriction
to any $S$-dimensional polydisk $Y\subset B\simeq \Delta^S\times \Delta^T$
in the universal deformation,
transversely intersecting the deepest
stratum
$\{0\}\times \Delta^T$,
is also a Kulikov model.

\begin{proposition}\label{reparam-action}
Fix $X_0/B_0$ as in Theorem \ref{smoothable}.
Then the formal torus
$\widehat{\bG}_m^S
$
acts naturally on $\hB$. The forgetful morphism
$
\hB\to\operatorname{Def}_{X_0}
$
is invariant under this action.
\end{proposition}

\begin{proof}
An $A$-valued point of $\hB$, 
for an Artinian local
$\bC$-algebra $A$, 
consists of a log smooth family
$
f\colon 
X_A\to({\rm Spec}\,A,\mathcal M_A),
$
a morphism
$
\alpha\colon 
{\rm Spec}\,A
\to
{\rm Spf}\,\bC[[u_s]]_{s\in S}
$
defining the log structure $\mathcal M_A$ on the base, and a
distinguished identification of its closed fiber with $X_0/B_0$.

Let
$
t=(t_s)_{s\in S}
\in\widehat{\bG}_m^S(A)
=(1+\mathfrak m_A)^S.
$
Define a new morphism $t\cdot\alpha$ by
$$
(t\cdot\alpha)^*(u_s)
=
t_s\alpha^*(u_s).
$$
The new morphism determines a log structure
$\mathcal M_A^t$ on ${\rm Spec}\,A$.
Multiplication of the defining 
chart by the units $t_s$
induces an isomorphism
$$
\iota_t\colon
({\rm Spec}\,A,\mathcal M_A)
\xrightarrow{\sim}
({\rm Spec}\,A,\mathcal M_A^t)
$$
whose underlying morphism of schemes is the identity. We define
$$
t\cdot(f,\alpha)
=
(\iota_t\circ f,t\cdot\alpha).
$$
Since $t_s\equiv1$ mod $\mathfrak m_A$,
the isomorphism
$\iota_t$ restricts to the identity on the closed log point
$B_0$. Thus the distinguished 
identification of the closed fiber
with $X_0/B_0$ is preserved.

This construction is functorial in $A$ and compatible with
multiplication of 
$t_1, t_2\in \widehat{\bG}_m^S$. So,
by pro-representability, it
defines an action
on $\hB={\rm Def}_{X_0/B_0}$. 
The underlying family of schemes
$
X_A\to{\rm Spec}\,A
$
is unchanged, so the forgetful morphism to
$\operatorname{Def}_{X_0}$ is
$\widehat{\bG}_m^S$-invariant.
\end{proof}

Let
$
D_s\coloneqq (X_0)_{\rm sing}^s
$
and suppose that $D_s$ is connected 
for every $s\in S$.\footnote{This is 
a natural condition to impose.
Suppose $d$-semistability 
(Def.~\ref{defn:dss}) holds
for an $S$-colored, polync variety
$X_0$. Then 
the same holds for an extended
color set $S'$,
indexed by the connected components
of $(X_0)_{\rm sing}^s$ for all $s\in S$.
Furthermore, now $(X_0)_{\rm sing}^{s'}$
is connected, for all $s'\in S'$ and
by Corollary \ref{universal}, we still
have a universal smoothing, but this
time with more smoothing parameters.} Let
$$
\cT^1_{X_0}
=
\cE xt^1(\Omega^1_{X_0},\cO_{X_0})
$$
be the sheaf of first-order local 
deformations of $X_0$.
As in \cite[p.~426]{olsson}, the choice
of $d$-semistable log structure 
$\cM_0$ corresponds to a trivialization 
of $\cT^1_{X_0}\simeq \bigoplus_{s\in S}
\cO_{D_s}$ along every $D_s$.
By the connectedness of $D_s$,
the space of such trivializations is a torsor
over $\bG_m^S$.

It follows that we have a homomorphism
\begin{align}\label{rho}
\rho\colon {\rm Aut}^0(X_0)\to \bG_m^S
\end{align}
where ${\rm Aut}^0(X_0)$ acts by pullback
on trivializations of $\cT_{X_0}^1$. 
Note that $\ker \rho = {\rm Aut}^0(X_0,\cM_0)$
and
${\rm im}\,\rho\subset \bG_m^S$
is a subtorus.
Let $$\bG \subset \bG_m^S$$ 
be a complementary subtorus to ${\rm im}\,\rho$.
We have the following proposition:

\begin{proposition}\label{reparam-quotient}
Continue in the setting of Proposition
\ref{reparam-action} under the additional
hypothesis that $(X_0)_{\rm sing}^s$ is connected
for all $s\in S$. The action
of $\widehat{\bG}$ 
on $\hB$ is free, and the
quotient
$
\hB/\widehat{\bG}
$
exists. It is the base of a semi-universal,
$d$-semistable deformation of $X_0$. 
This deformation
is furthermore universal when
$
\bG=\bG_m^S.
$
\end{proposition}

\begin{proof}
Let $\Theta_{X_0}$ denote the ordinary tangent sheaf
of $X_0$ and let $\Theta^{\rm log}_{X_0/B_0}$ denote
the logarithmic tangent bundle.
We have an exact sequence
\begin{align}\label{exact}
0\to \Theta_{X_0/B_0}^{\rm log}\to 
\Theta_{X_0}\to \cT_{X_0}^1\to 0. \end{align} 
This is the logarithmic analogue of the relative 
tangent sequence for a normal crossings variety; see
\cite[Prop.~5.4 and Ex.~5.7]{kato}. 
The same calculation as {\it loc.cit.}~applies
to the product charts defining a polync variety.

Since every $D_s$ is connected, the
$d$-semistability gives
$$\textstyle 
H^0(X_0,\cT^1_{X_0})
\simeq
\bigoplus_{s\in S}H^0(D_s,\cO_{D_s})
\simeq\bC^S.
$$
Under this identification, the coboundary map
associated to the exact sequence (\ref{exact}) is
$$
\partial\colon
\bC^S\to
H^1(X_0,\Theta^{\rm log}_{X_0/B_0}).
$$
The differential at the identity of the
$\widehat{\bG}_m^S$-action of Proposition
\ref{reparam-action} is precisely $\partial$. 
Indeed,
an infinitesimal coordinate rescaling changes the
chosen trivializations of $\cT^1_{X_0}$ and
$\partial$ measures whether this change is 
induced by an infinitesimal automorphism 
of $X_0$. This is
also the infinitesimal form of the 
torsor description
in \cite[p.~426]{olsson}.
Thus, $${\rm Lie}(\bG) \subset {\rm Lie}(\bG_m^S)=\bC^S$$ maps isomorphically
onto its image under $\partial$, whereas
${\rm Lie}({\rm im}\,\rho)$ forms the
kernel
of $\partial$.

On the level of tangent spaces,
the forgetful map from logarithmic deformations
to deformations is given by the middle downward
morphism of
\[
\xymatrix{
0 \ar[r] &
H^1(X_0,\Theta^{\rm log}_{X_0/B_0})
\ar[r] \ar[d] &
T_0\hB
\ar[r] \ar[d] &
\bC^S
\ar[r] \ar[d]^{\simeq} &
0 
\\
0 \ar[r] &
H^1(X_0,\Theta_{X_0})
\ar[r] &
T_0{\rm Def}_{X_0}
\ar[r] &
H^0(X_0,\cT^1_{X_0})
\ar[r] &
0.
}
\]
Here the bottom short exact 
sequence is the local-to-global sequence.
The connectedness of $D_s$ 
ensures
the rightmost downward arrow
is an isomorphism, and in turn,
the commutativity is what
ensures the surjectivity of
$T_0{\rm Def}_{X_0}\to H^0(X_0,\cT^1_{X_0})$.

It follows
that \begin{align}\label{eq}\ker(T_0\hB\to T_0{\rm Def}_{X_0})
\simeq {\rm im}(\partial) \simeq 
{\rm Lie}(\bG).\end{align}
The formal deformation space $\hB$ 
is smooth by Theorem \ref{smoothable}. 
It follows from (\ref{eq})
that the action of $\widehat{\bG}$
on $\hB$ is free, as infinitesimally,
it acts by translations on $H^1(X_0,\Theta_{X_0/B_0}^{\rm log})\subset T_0\hB$.

Hence, there exists a smooth
slice $\hB\simeq \widehat{\bG}\times \hC$
and $T_0\hC\to T_0{\rm Def}_{X_0}$ is injective. 
The restriction of the universal
family $\hX\to \hB$ to $\hC$ is complete
for $d$-semistable deformations of $X_0$
as all such deformations
lift to logarithmic deformations
of $(X_0,\cM_0)/B_0$ which can
in turn be translated by $\widehat{\bG}$
to land in $\hC$.
Thus, the restriction of $\hX$ to $\hC$
is semi-universal
for the $d$-semistable deformations
of $X_0$.

When $\bG=\bG_m^S$  
we have ${\rm ker}(\partial)= 0$. Hence
$h^0(X_0,\Theta_{X_0})=
h^0(X_0,\Theta_{X_0/B_0}^{\rm log})$
from the long exact sequence of (\ref{exact}). 
The same equality holds for the generic
fiber of $\hX\to \hB$, which is smooth. 
Since, $h^0(X_b,\Theta_{X_b}^{\log})$ is constant
(see Cor.~\ref{universal}) and
$h^0(X_b,\Theta_{X_b})$ 
is upper semicontinuous, 
we conclude that
$h^0(X_b,\Theta_{X_b})$ is constant 
for all $b\in \hB$. Then Wavrik's criterion
\cite[Thm.~4.2]{wavrik} applies again to upgrade
the semi-universal deformation over $\hC$ to a
universal deformation.
\end{proof}

We note that one may also slice the universal 
analytic log smooth deformation family,
giving a polydisk $C\subset B\simeq 
\Delta^S\times \Delta^T$, where $C\simeq \Delta^S\times \Delta^{T'}$ with
$|T|=|T'|+ |S|-\dim ({\rm im}\,\rho)$.

\begin{definition}\label{sm-comp}
We write $C = {\rm smDef}_{X_0}\subset 
{\rm Def}_{X_0}$ for the semi-universal
$d$-semistable deformation of $X_0$.
It is the {\it smoothing component}
of ${\rm Def}_{X_0}$.
\end{definition}

The term {\it smoothing component}
is justified by the fact that the general
fiber of $X\to B$ is smooth with trivial log
structure. Thus, at the general fiber, the
logarithmic and classical deformation problems
coincide. It follows that the forgetful map 
$B\to C\to {\rm Def}_{X_0}$
dominates a component
of the Kuranishi space. 
Since $T_0C\hookrightarrow 
T_0{\rm Def}_{X_0}$ injects, 
this component is, on the level of germs,
identified with $C$.

\section{Polysnc $K$-trivial surfaces}

We assume throughout this section
that $X_0$ is a proper
algebraic space (e.g.~that
the components of the normalization 
are proper varieties). By Definition
\ref{colorable}, we have:

\begin{observation}\label{obs}
A $2$-dimensional,
$S$-colored polysimplicial complex 
$\Gamma$ is a 
polyhedral complex built
from triangles and squares, admitting
an $S$-coloring of the edges $E(\Gamma^{[1]})\to S$
such that:
\begin{enumerate}
\item all three sides of a triangle are the same color,
\item opposite sides of a square are the same
color, and
\item each square uses exactly two colors.
\end{enumerate}
\end{observation}

\begin{definition}
A {\it poly(s)nc log Calabi--Yau pair} $(V,D)$ is a smooth
proper variety $V$, together with an anticanonical
divisor $D\in |-K_V|$, such 
that $D$ is reduced and poly(s)nc.
\end{definition}

It follows easily from adjunction
that a poly(s)nc variety
with $\omega_{X_0}\simeq \cO_{X_0}$
is a union of poly(s)nc log CY pairs, glued
along anticanonical divisors, so that
no boundary remains.

In particular, since poly(s)nc singularities
in dimension $1$ are just nodes/(s)nc singularities,
a poly(s)nc log CY surface $(V,D)$ is just
a smooth proper surface with a
reduced, (s)nc
anticanonical divisor $D\in |-K_V|$,
also called a Looijenga pair. See \cite{friedman2015} for a survey
on their geometry. Such pairs break
into three natural classes:

\begin{enumerate}
\item\label{case1} When
$D$ has nodes, $V$ is rational and
$D$ is a cycle of smooth rational curves.
\item\label{case2} When $D$ is smooth, either 
$V$ is rational and $D$ is a smooth
elliptic curve, or
\item\label{case3} $V$ is ruled
over an elliptic curve, and $D$
is a disjoint union of two elliptic
sections.
\end{enumerate}

\begin{proposition}\label{homeo-type}
Let $X_0$ be a $d$-semistable,
$S$-colored, polync surface with 
$\omega_{X_0}\simeq \cO_{X_0}$. Then
$\Gamma(X_0)$ is a point, an interval,
a circle, or a triangle-squarulation
of the two-sphere $\bS^2$ or the two-torus
$\bT^2$. For $\bS^2$ or the interval,
the general fiber $X_t$ of the universal
log smooth
deformation $X\to \Delta^{S}\times \Delta^T$
is a K3 surface, and for $\bT^2$
or the circle, it is an abelian surface.
\end{proposition}

Observe that case (\ref{case1}) above occurs
exactly when $\Gamma(X_0)$ is $2$-dimensional, and cases (\ref{case2}, \ref{case3}) occur exactly
when $\Gamma(X_0)$ is $1$-dimensional.

\begin{proof}
We will reduce to the same statement
but for nc varieties. 
When $\dim_\bR \Gamma(X_0)\leq 1$,
nc and polync are the same,
and the result follows from 
\cite{kul, pp}, 
see also \cite[Table 1, p.~12]{prog}.

Consider the diagonal
$\Delta\hookrightarrow \Delta^S$
and then further embed
$$\Delta \hookrightarrow 
\Delta^S\simeq \Delta^S\times \{0\} \hookrightarrow 
\Delta^S\times \Delta^T$$
into the base of the universal
(log smooth) deformation. Let $X_\Delta\to \Delta$ be the pulled back deformation.
The local equation of the singularities of the universal family $X\to \Delta^S\times \Delta^T$ is given by (\ref{product2}).
The base change to $\Delta$
corresponds to setting
$u_s=u$ for all $s\in S$, where $u$
is a coordinate on $\Delta$ (and
$v_t=0$, for all $t\in T$).
It follows that the local
equation of $X_\Delta\to \Delta$ 
is \begin{align}\label{conifold}
x_iy_i=x_jy_j=u\end{align} for each
square of $\Gamma(X_0)$ 
colored by $s(i)\neq s(j)$,
and is $x_iy_iz_i=u$ for each triangle
of $\Gamma(X_0)$. Then analytically,
$X_\Delta\to \Delta$ admits a small,
crepant resolution $\wX_\Delta\to \Delta$
by taking either resolution of the
conifold singularity (\ref{conifold}).
The central
fiber $\wX_0$ is still an algebraic space,
and the resulting family 
$\wX_\Delta\to\Delta$ is a Kulikov model,
now over a $1$-parameter base.

Furthermore, the homeomorphism type of $\Gamma(X_0)$ is the same as 
$\Gamma(\wX_0)$, with the only difference
being that all squares in $\Gamma(X_0)$
have been subdivided into two triangles. 
The result now reduces to \cite{kul, pp, prog}, which implies the 
proposition for a $1$-dimensional base.
\end{proof}

The argument of Proposition \ref{homeo-type}
is quite general. It shows that the
homeomorphism type of the dual complex 
of a Kulikov model
is the same, whether one considers a
multiparameter degeneration, or a resolution of a $1$-parameter 
degeneration hitting
the origin. \smallskip

We now briefly overview how
to generalize the classical theory of
Kulikov models and their period map
\cite{fs, ae} to the polync setting.

\begin{definition}\label{num-dss} Let 
$X_0=\bigcup_i (V_i,\sum_j D_{ij})$ be an
$S$-colored, poly(s)nc
union of log CY surface pairs,
glued along isomorphisms
$D_{ij}\to D_{ji}$. 
We say that $X_0$ is {\it numerically
$d$-semistable} if $\textstyle {\cE xt}^1(\Omega^1_{X_0}, \cO_{X_0})\simeq 
\bigoplus_{s\in S}
\cL_s$ where $\cL_s\in {\rm Pic}^0((X_0)_{\rm sing}^s)$
are numerically trivial for all $s\in S$.
If additionally $\Gamma(X_0)\simeq \bS^2$,
we call $X_0$ a {\it poly(s)nc Type III K3 surface}.
And if $\Gamma(X_0)$ is an interval,
we call $X_0$ an {\it snc Type II K3 surface}.
\end{definition}

Note that any polync Type II K3 surface
is simply snc. For convenience, 
we will henceforth assume
that $X_0$ is polysnc and 
therefore, that polysimplices 
are embedded into $\Gamma(X_0)$, 
i.e.~not 
glued to themselves along faces. Most
of the results here 
continue to hold in the more
general setting, but for example,
formula (\ref{tp}) and Definition
\ref{period} below must be adjusted.

\begin{exercise}\label{exer}
Let $X_0$
be a polysnc, Type III K3 surface.
Then $X_0$ is numerically $d$-semistable
if for all irreducible
double curves $\bP^1\simeq
D_{ij}\subset V_i\cap V_j$, we have
\begin{align}
    \label{tp}
D_{ij}^2+D_{ji}^2=\threepartdefwhen{0}{\textrm{the associated edge of }\Gamma(X_0)\textrm{ lies on two squares},}{-1}{\textrm{the associated
edge of }\Gamma(X_0)\textrm{ lies on a square and a triangle},}{-2}{\textrm{the associated edge of }\Gamma(X_0)\textrm{  lies on two triangles}.}
\end{align}
Here by convention $D_{ij}\subset V_i$
and $D_{ji}\subset V_j$.
This formula is the polysnc
analogue of the classical
``triple point formula'' for $1$-parameter
Kulikov models \cite[p.~8]{prog}.
\end{exercise}

The condition that $X_0$ 
be numerically $d$-semistable
is a combinatorial condition
on $X_0$ in the sense that the locally
trivial deformations of $X_0$ do not
affect this property.

\begin{definition}\label{period}
The {\it numerically Cartier classes}
on a polysnc Type III K3 surface $X_0$ are
$$\widetilde{\Lambda}(X_0)\coloneqq \ker\left(\textstyle 
\bigoplus_i H^2(V_i,\bZ)\to \bigoplus_{i<j} H^2(D_{ij},\bZ)\right).$$

Following \cite[Sec.~3]{fs} or 
\cite[Sec.~4]{ae}, we may define
a {\it period map} $\varphi_{X_0}\colon \widetilde{\Lambda}(X_0)\to \bC^*$
as follows: Since each component $V_i$
is rational and each double locus
$D_{ij}\subset V_i\cap V_j$ is isomorphic
to $\bP^1$,
we have $H^2(V_i,\bZ)={\rm Pic}(V_i)$
and 
$H^2(D_{ij},\bZ)={\rm Pic}(D_{ij})$.

Let $\alpha = (\alpha_i)\in \widetilde{\Lambda}(X_0)$. Then each
$\alpha_i$ defines canonically
a line bundle $\cL_i\to V_i$
and the condition that $\alpha_i\cdot D_{ij}=\alpha_j\cdot D_{ji}$ implies
that we may glue $\cL_i$ and $\cL_j$
along $D_{ij}$. In this manner, we may
successively glue the $\cL_i$ together
until only one component $V_n$ remains.
Then this glued line bundle
$\cL\to \bigcup_{i\neq n} V_i\subset X_0$ 
and the line bundle $\cL_n\to V_n$ 
differ by some $$\textstyle
\varphi_{X_0}(\alpha)\in \bC^*
\simeq {\rm Pic}^0(\sum_j D_{nj})$$
upon restriction
to the anticanonical cycle 
$D_n=\sum_j D_{nj}\in |-K_{V_n}|$. 
The construction in Type II is similar, see
\cite[Constrs.~4.2, 4.3]{ae}.
\end{definition}

The element 
$\varphi_{X_0}(\alpha)\in \bC^*$ can be
checked to be independent of the choice
of a ``last'' component $V_n$. It only depends
on a choice of orientation of $\bS^2$,
which in turn orients the anticanonical
cycle on any component $V_i\subset X_0$
(the identification ${\rm Pic}^0(D_i)\simeq \bC^*$
requires a choice of orientation of the cycle
$D_i=\sum_j D_{ij}\subset V_i$ of rational curves).
We have an isomorphism 
$${\rm Pic}(X_0)\simeq 
\ker(\varphi_{X_0}).$$ 
That is, $\varphi_{X_0}$ measures
the obstruction to lifting
a numerically Cartier class to
a Cartier class.

Recall that $\Gamma(X_0)$ is an 
$S$-colored graph. Let
${\rm I}(X_0)$ be the {\it intersection
complex}, which is the polyhedral
decomposition of $\bS^2$ dual to $\Gamma(X_0)$.
The vertices of ${\rm I}(X_0)$
are $3$- and $4$-valent, and
the $1$-skeleton ${\rm I}(X_0)^{[1]}$
is also $S$-colored.
The color $s$ subgraph
${\rm I}_s(X_0)^{[1]}\subset 
{\rm I}(X_0)^{[1]}$ is a graph
with trivalent and bivalent vertices,
corresponding respectively
to triangles and squares
which use the color $s$.
See Figure \ref{fig:rhomb-dual}
for an example.

\begin{definition}\label{slab}
A {\it slab of color} $s\in S$ is a
union of components
$\bigcup_{i\in G} V_i$ corresponding
to a connected component of ${\rm I}(X_0)\setminus {\rm I}_s(X_0)^{[1]}$.
\end{definition}

\begin{remark}
The notion of a slab generalizes
to $S$-colored, polysnc varieties
of any dimension, by considering
the codimension $1$ faces
of ${\rm I}(X_0)$ dual to edges
in $\Gamma(X_0)^{[1]}$ of color $s$.
\end{remark}

\begin{example}
If $X_0$ is snc and $S= \{s\}$,
then the slabs of color $s$ are just
the irreducible components $V_i\subset X_0$
because every codimension
$1$ cell of ${\rm I}(X_0)$ 
has color $s$,
and so the connected components
in Definition \ref{slab} are the
maximal dimension cells of ${\rm I}(X_0)$.
\end{example}

It follows from
the local form of a semistable morphism
(\ref{product2}) that the slabs of color $s$
deform to the irreducible components
over the coordinate hyperplane $V(u_s)\subset \Delta^S\times \Delta^T$,
because the singular loci of every
other color $s'\neq s$ are smoothed.

On the total space $X$, we have 
a line bundle $\cL_G\coloneqq \cO_X(V_G)$
where $V_G$ is the
irreducible component over $V(u_s)$
corresponding to the slab indexed by $G\subset \{\textrm{components of }X_0\}$.
Thus the condition
that $X_0$ is smoothable 
implies that $\cO_X(G)\vert_{X_0}$
defines an element of ${\rm Pic}(X_0)$.
The numerical Cartier class of $\cO_X(V_G)\vert_{X_0}$ is easy
to compute; it is $$\xi_G = \!\!\!\!\!
\sum_{\substack{\textrm{oriented edges }e_{ij}\textrm{ of} \\
\textrm{color }s \textrm{ bounding }G}}
\!\!\!\!\!
[D_{ji}] - [D_{ij}]\in  \widetilde{\Lambda}(X_0).$$
So to be smoothable, we certainly need
$\varphi_{X_0}(\xi_G)=1$ for any
slab $G$. And conversely:

\begin{proposition} \label{dss-criterion}
Let $X_0$ be a polysnc 
Type III K3 surface. Then $X_0$
is $d$-semistable if and only if
$\varphi_{X_0}(\xi_G)=1$ for all
slabs $G$ of all colors $s\in S$.
\end{proposition}

\begin{proof} When $S=\{s\}$, this
is \cite[Prop.~4.4]{ae}.
Fix a color $s\in S$ and set
$C_s\coloneqq (X_0)_{\rm sing}^s$.
Since $X_0$ is numerically
$d$-semistable, $\cL_s$ has degree zero
on every irreducible component of $C_s$.
As these components are rational,
$\cL_s$ is determined by its gluing
monodromy, which defines a character
$$
H_1({\rm I}_s(X_0)^{[1]},\bZ)\to\bC^*.
$$
The same computation as in the snc case
shows that the value of this character
on the oriented boundary of a slab $G$
is $\varphi_{X_0}(\xi_G)$.
The oriented boundaries of the slabs
generate
$H_1({\rm I}_s(X_0)^{[1]},\bZ)$.
Consequently, $\cL_s$ is trivial if and
only if $\varphi_{X_0}(\xi_G)=1$ for
every slab $G$ of color $s$.
Applying this independently for every
$s\in S$ proves the result.
\end{proof}

\begin{remark} If $X_0$ is smoothable,
then by base changing
to a $1$-parameter family and resolving
to a $1$-parameter Kulikov model
as in Proposition \ref{homeo-type}, we 
get a morphism $\wX_0\to X_0$
which defines a bijection on components.
For two of 
the irreducible components $\wV_i\to V_i$ 
corresponding to a vertex 
$v_i\in \Gamma(X_0)^{[0]}$
adjacent to a square, this
modification restricts
to the blow-up of a node of the anticanonical
cycle $D_i=\sum_j D_{ij}\in |-K_{V_i}|$.

Thus, via pullback, we get an inclusion
$\iota\colon \widetilde{\Lambda}(X_0)
\hookrightarrow
\widetilde{\Lambda}(\wX_0)$
and it is easily seen from Definition
\ref{period} that $\varphi_{X_0} = \varphi_{\wX_0}\circ \iota$. 
So the period $\varphi_{X_0}$
can be computed via the resolution
to a Kulikov model
of a $1$-parameter base change.
\end{remark}

\begin{corollary}\label{def-to-dss}
Any polysnc Type III K3 surface $X_0$ admits
a locally trivial deformation to one
which is $d$-semistable (and in 
particular, smoothable).
\end{corollary}

\begin{proof}
Build an snc surface $\wX_0\to X_0$
by subdividing every square of $\Gamma(X_0)$
into two triangles, performing
the corresponding corner blowups on
components of $X_0$. Then $\wX_0$
is an snc
Type III K3 surface by Exercise
\ref{exer}. Thus,
there is a locally
trivial deformation $\wY_0$ of $\wX_0$
which is $d$-semistable,
e.g.~by taking the 
``standard blow-ups and gluings'' 
\cite[Lem.~2.8]{ghk1},
\cite[Prop.~7.11]{ae} for which $\varphi_{\wY_0}=1$.
Let $\wY_\Delta\to \Delta$ be a smoothing
of $\wY_0$. The diagonal subdivision
of each square corresponds to a double
curve $E_{ij}\subset \wY_0$ satisfying
$E_{ij}^2=
E_{ji}^2=-1$. 

Thus the
$E_{ij}$ are analytically 
contractible in the 
total space $\wY_\Delta\to \Delta$.
We let $Y_\Delta\to \Delta$ be the 
contraction, and $Y_0$ be its central fiber.
We have that $Y_0$ is a locally trivial
deformation of $X_0$. Furthermore,
$Y_0$ is $d$-semistable by
Proposition \ref{dss-criterion},
the fact that $\wY_0$ is, and 
the identity $\varphi_{Y_0} = \varphi_{\wY_0}\vert_{\widetilde{\Lambda}(Y_0)}$ on period homomorphisms.
\end{proof}

Corollary \ref{def-to-dss} grants
the existence of a {\it smoothable}
polysnc Type III K3 surface, within
the locally trivial deformation class 
of a fixed numerically 
$d$-semistable one. This justifies
why the examples in Section
\ref{examples} are smoothable.

We now make some basic
parameter counts,
to determine the dimension
of the space of locally trivial,
$d$-semistable deformations.
First, we have the
``conservation of charge'':

\begin{proposition}\label{charge}
Let $X_0$ be a polync Type III (or II) K3 surface.
Then $ \sum_i Q(V_i,\sum_j D_{ij})=24$
where $Q(V,D)\coloneqq \chi_{\rm top}(V\setminus D)$ is the ``charge'' of $(V,D)$.
\end{proposition}

\begin{proof}
Two simple proofs are as follows:
We can reduce to the snc case, by
taking an nc ``resolution''
$\wX_0\to X_0$ as in Proposition
\ref{homeo-type} or Corollary 
\ref{def-to-dss}. Then the log CY
components
$(\wV_i,\wD_i)\subset \wX_0$ 
are all corner blow-ups
of the components 
$(V_i,D_i)\subset X_0$
and thus 
$\chi_{\rm top}(\wV_i\setminus \wD_i)=\chi_{\rm top}(V_i\setminus D_i)$.
So the formula follows from the $1$-parameter
case \cite[Prop.~3.7]{fm}.

Alternatively, we note
that after a topologically trivial deformation, $X_0$ is smoothable
by Corollary \ref{def-to-dss}. 
Let $X_t$ be the smooth general
fiber. It is a K3 surface
by Proposition \ref{homeo-type}.
Then by \cite[Sec.~5.2]{survey},
there is a multivariable Clemens
collapsing map $c_t\colon X_t\to X_0$
which is a $(\bS^1)^r$-bundle over the 
codimension $r$ singular strata
of $X_0$. It follows from additivity 
and  multiplicativity
of the Euler characteristic
that $24 = \chi_{\rm top}(X_t)=
\sum_i \chi_{\rm top}(V_i\setminus D_i)$.
\end{proof}

\begin{proposition}\label{params}
Let $X_0$ be a $d$-semistable
polysnc Type III K3 surface. The naive
parameter count
of the space of $d$-semistable,
locally trivial deformations is
$20-|S|.$
\end{proposition}

\begin{proof}
The parameter count of the space of all locally
trivial deformations, possibly not
$d$-semistable, is given by
$e+24-2v$ where $e$ is the number
of edges of $\Gamma(X_0)$ and $v$
is the number of vertices. This is 
because there is a $\bC^*$-worth 
of regluings
$D_{ij}\to D_{ji}$
 of each double
curve,
each non-toric
blow-up of a toric pair 
varies in a copy of $\bC^*$,
and the number of non-toric
blow-ups is $\sum_i Q(V_i,D_i)=24$ by
Proposition \ref{charge}.
This gives $e+24$ parameters,
but we must quotient by 
re-parameterizations, corresponding to the automorphism groups $(\bC^*)^2$
of toric models $(\oV_i,\oD_i)$
of each component $(V_i,D_i)$,
see e.g.~\cite[Sec.~7.2]{ae}.

Imposing $d$-semistability then
amounts to $ \sum_{s\in S} (n_s-1)$
conditions, where $n_s$ is the number
of slabs of color $s$, i.e.~the number
of connected components of 
${\rm I}(X_0)\setminus {\rm I}_s(X_0)^{[1]}$. 
We subtract $1$
from $n_s$
because the sum of all slab classes
of color $s$ is 
$0\in \widetilde{\Lambda}(X_0)$.
Thus, we get a naive parameter count
of $$\textstyle e+24-2v-\sum_{s\in S}(n_s-1)$$ for the $d$-semistable surfaces in
the deformation class of $X_0$. 
Recall that the color $s$ subgraph
${\rm I}_s(X_0)^{[1]}$ 
is a graph of valencies $2$, $3$
with the vertices of valence $2$
corresponding to squares with
opposite sides of color $s$,
and vertices of valence $3$ corresponding
to triangles of color $s$. 
We deduce from Gauss-Bonnet
that $n_s = 2+\tfrac{1}{2}n_{\triangle, s}$ where $n_{\triangle,s}$ is the number of triangles
of color $s$ in $\Gamma(X_0)$. 
In turn, our parameter
count becomes $$e+24-2v-|S|-\tfrac{1}{2}n_\triangle$$
where $n_\triangle$ is the total number of triangles of $\Gamma(X_0)$.
Let $n_\square$
be the number of squares
in $\Gamma(X_0)$. Again by Gauss--Bonnet,
we have
$v-e+n_\triangle +n_\square =2.$
Furthermore,
$3n_\triangle + 4n_\square = 2e$.
Eliminating $n_\square$ gives
$\textstyle
v=2+\tfrac{1}{2}e-\tfrac{1}{4}n_\triangle$
and finally substituting into the parameter
count above gives a final answer of
$20-|S|.$
\end{proof}

The parameter count in Proposition
\ref{params} is valid, if and only if
the signed restriction map
$\bigoplus_{i} H^2(V_i,\bQ)\to\bigoplus_{i<j} H^2(D_{ij},\bQ)$ is surjective. 

\begin{proposition}\label{bigger}
Let $X\to B\simeq 
\Delta^S\times \Delta^T$ be 
the universal log smooth
deformation of a $d$-semistable
polync K3 surface,
endowed with a log structure of
semistable type. Then $B$ is relatively
smooth over $\Delta^S$ of relative
dimension $|T|=20$.
\end{proposition}

\begin{proof}
The relative
tangent space of $B/\Delta^S$
is $T_0(B/\Delta^S)\simeq 
H^1(X_0,\Theta^{\rm log}_{X_0/B_0})$.
Contraction with a logarithmic
non-vanishing two-form $\eta$ gives
an isomorphism
$$
H^1(X_0,\Theta^{\rm log}_{X_0/B_0})
\xrightarrow{\;\lrcorner\,\eta\;}
H^1(X_0,\Omega^{\rm log}_{X_0/B_0}).
$$
The latter space satisfies
$
H^1(X_0,\Omega^{\rm log}_{X_0/B_0})
\simeq\bC^{20}
$
by the constancy
of the logarithmic Hodge number, since the
nearby smooth fiber is a K3 surface.

Since $B$ is
smooth over $\Delta^S$ by Theorem
\ref{smoothable}, it follows that
$\dim(B/\Delta^S)=20$.
\end{proof}

\begin{remark}
The same argument as Proposition
\ref{bigger} applies more
generally in the setting of Theorem \ref{smoothable}
and Corollary \ref{universal}, and showing
that $B\simeq \Delta^S\times \Delta^T$ with
$|T|=h^{d-1,1}(X_b)$,
$d= \dim X_b$ for a smooth $K$-trivial
fiber $X_b$ of the universal deformation 
$X\to B$.
\end{remark}

A potentially confusing aspect of 
Proposition \ref{bigger} is that
$\dim B=20+|S|$, not $20$. This is because
${\rm Def}_{X_0/B_0}$ is the
deformation functor of the logarithmic
morphism, rather than the ordinary
deformation functor of $X_0$.
The forgetful map
${\rm Def}_{X_0/B_0}\to {\rm Def}_{X_0}$
can have positive-dimensional fibers,
see Propositions \ref{reparam-action},
\ref{reparam-quotient}. In our case, 
the component 
${\rm smDef}_{X_0}\subset {\rm Def}_{X_0}$
in which ${\rm Def}_{X_0/B_0}$ lands 
(Def.~\ref{sm-comp})
satisfies $${\rm smDef}_{X_0}\simeq \Delta^S\times \Delta^{20-|S|+r}.$$
Here ${\rm Aut}^0(X_0)$ is necessarily
an algebraic torus, 
and ${\rm Aut}^0(X_0)\simeq (\bC^*)^r$
by the infinitesimal injectivity of (\ref{rho}), which in turn follows from $$h^0(X_0,\Theta^{\rm log}_{X_0/B_0})=h^0(X_0, \Omega^{\rm log}_{X_0/B_0})=h^{1,0}(X_b)=0.$$

\begin{proposition}\label{monodromy-cone}
Let $T_s\colon H^2(X_t,\bZ)\to H^2(X_t,\bZ)$
be the Picard--Lefschetz transform
about $V(u_s)$
on the general fiber of the universal
smoothing $X\to {\rm smDef}_{X_0}=\Delta^S\times \Delta^{T'}$
of a $d$-semistable polysnc
K3 surface $X_0$.
Let $N_s\coloneqq \log T_s$.
There is an isotropic vector $\delta\in
H^2(X_t,\bZ)$ and a collection 
of vectors $\lambda_s\in \delta^\perp/\delta$, 
$s\in S$, such that:

\begin{enumerate}
\item $N_s(x) = (x\cdot \lambda_s)\delta-(x\cdot \delta)\lambda_s$, 
\item $\lambda_s\in \cC_\delta^+$
lies in the rational closure
of the positive cone $\cC_\delta\subset \delta^\perp/\delta \otimes \bR$,
\item $\lambda_s\cdot \lambda_s = n_{\triangle,s}$ is the number of triangles
in $\Gamma(X_0)$ of color $s\in S$,
\item $\lambda_s\cdot\lambda_{s'} = n_{\square,s,s'}$ is
the number of squares
in $\Gamma(X_0)$ colored by $s\neq s'\in S$,
\item if $\{\lambda_s\}_{s\in S}$
are linearly independent and $(X_0)_{\rm sing}^s$ is connected for all $s\in S$, ${\rm smDef}_{X_0}\simeq \Delta^{S}\times\Delta^{20-|S|}$ and an
appropriate boundary period map is a
local isomorphism.
\end{enumerate}
\end{proposition}

We call the $\bR_{\geq 0}$-span of $\{\lambda_s\}_{s\in S}$
in $\cC_\delta^+$ the {\it monodromy cone}
of the degeneration.

\begin{proof}[Sketch.] 
Item (1) follows from generalities
on the weight filtration of a degeneration
of Hodge structures of weight
$2$ and type $(1,h,1)$ over a polydisk.

Items (2) and (3) follow from taking
a generic $1$-parameter degeneration
$X_s\to \Delta_s\subset \Delta^S\times \Delta^{T'}$
via base change along
$u_s=u$ and $u_{s'} = {\rm const}\in \Delta^*$ 
for $s'\neq s$ (and $v_t=0$). 
This is a $1$-parameter Kulikov
model transversely slicing the divisor
$V(u_s)$. 
So we may apply the result
\cite[Prop.~7.1]{fs},
which gives the formula $\lambda_s\cdot \lambda_s=
\#\{\textrm{triple points of }(X_s)_0\}$.
Here the key point
is that the triangles of color $s$
in $\Gamma(X_0)$
are in bijection with the triple
points of $(X_s)_0$. 

To prove (4), take a $1$-parameter
degeneration $X_{s,s'}\to \Delta_{s,s'}$ 
via the base change
$u_s=u_{s'}=u$ and $u_{s''}={\rm const}\in \Delta^*$
for $s''\neq s,s'$. The base changed 
degeneration is not Kulikov, but becomes 
Kulikov after passing to a crepant resolution
$\wX_{s,s'}\to \Delta_{s,s'}$
as in Proposition \ref{homeo-type}.
The resolution $\wX_{s,s'}$
corresponds to subdividing every square
in $\Gamma(X_0)$ associated to a 
conifold point $x_iy_i=x_jy_j=u$
with colors $\{s(i), s(j)\}=\{s,s'\}$
into two triangles---each such
square introduces
two triple points to $(\wX_{s,s'})_0$
in addition to those already on $(X_{s,s'})_0$ 
coming from triangles of color $s$ or $s'$. Note that the monodromy
over $\Delta_{s,s'}^*$ is the product
of the monodromies over $\Delta_s^*$
and $\Delta_{s'}^*$ because
$\pi_1((\Delta^*)^S\times \Delta^{T'})\simeq \bZ^S$.
Hence the logarithms 
of monodromy satisfy 
$N_{s,s'} = N_s+N_{s'}$.
It again follows
by applying \cite[Prop.~7.1]{fs} that
$$(\lambda_s+\lambda_{s'})\cdot 
(\lambda_s+\lambda_{s'}) = n_{\triangle,s} + 2n_{\square,s,s'}+n_{\triangle,s'}.$$
Expanding the left-hand side
and applying (3) gives (4).
We prove (5) in Lemma \ref{li}.
\end{proof}

\begin{lemma}\label{li} Suppose $X_0$ is
an $S$-colored, polync K3 surface
with each $(X_0)_{\rm sing}^s$ connected.
Then ${\rm Aut}^0(X_0)\simeq (\bC^*)^r$
where $r$ is the rank of the space of
linear relations amongst the $\{\lambda_s\}_{s\in S}$.

Furthermore, when $r=0$, 
there is a local holomorphic period map
$\Phi\colon {\rm smDef}_{X_0}\to X(\sigma)$
to a smooth 20-dimensional variety
which is a local isomorphism.
\end{lemma}

\begin{proof}
As in Proposition \ref{reparam-quotient}, the
long exact sequence of (\ref{exact}) contains
\[
0\to H^0(X_0,\Theta_{X_0})
\to \bC^S\xrightarrow{\partial}
H^1(X_0,\Theta^{\rm log}_{X_0/B_0}),
\]
where we have used
$H^0(X_0,\Theta^{\rm log}_{X_0/B_0})=0$.
The first map is the differential of
$\rho$ in (\ref{rho}).

Let $\pi\colon X\to B$ be the universal
log smooth deformation and let
$B^\circ\subset B$ be its smooth locus.
The Deligne extension of
$R^2\pi^\circ_*\underline{\bZ}_{X^\circ}
\otimes \cO_{B^\circ}$ 
carries a logarithmic
Gauss--Manin connection $\nabla$ and an
extended Hodge filtration $\cF^\bullet$.
In our case,
we have an isomorphism
\begin{align}\label{ks}
H^1(X_0,\Theta^{\rm log}_{X_0/B_0})\simeq {\rm Hom}(H^0(X_0, \omega_{X_0}^{\rm log})\to H^1(X_0, \Omega^{\rm log}_{X_0/B_0})) = {\rm Hom}(\cF^2_0, \cF^1_0/\cF^2_0)
\end{align}
given by contraction with a logarithmic two-form,
see Proposition \ref{bigger}. The associated
graded of the logarithmic connection gives the
logarithmic Higgs field 
$${\rm gr}_\cF(\nabla)=\textstyle \sum {\rm gr}_\cF^p(\nabla)\in  
\bigoplus_{p+q=2} {\rm Hom}
(\cF^p/\cF^{p+1}, \cF^{p-1}/\cF^p)
\otimes \Omega^{\rm log}_B.$$
Furthermore, we have
as in the classical calculation relating the Gauss--Manin connection
to the Kodaira--Spencer map,
an equality 
\begin{align}\label{eq1}\iota_{u_s\partial_s}\circ
{\rm gr}^2_\cF(\nabla) = \partial(e_s)\end{align}
under the isomorphism (\ref{ks}). Here $u_s\partial_s\in \Theta_{B,\,0}^{\rm log}$
is a logarithmic tangent vector at $0\in B$
and $\iota$ denotes the contraction.

If $N_s=\log T_s$, then $N_s$ defines a weight $-2$
endomorphism of the mixed Hodge structure
$(\bV_0, \cF^\bullet_0, W_\bullet)$ on the 
logarithmic de Rham
cohomology over $0\in B$.
By \cite[Prop.~II.3.11]{deligne-reg},
\begin{align}\label{eq2}
{\rm Res}_{V(u_s)}(\nabla) = 
\iota_{u_s\partial_s} \circ \nabla
=-(2\pi i)^{-1}N_s
\end{align}

In the Type III case, $W_0=\bZ\delta$
and $W_2/W_0=\delta^\perp/\delta$. The
weight-graded part of $N_s$ is therefore
\begin{align*}
{\rm gr}_2^W(N_s)\colon
\delta^\perp/\delta&\to\bZ\delta, \\
x&\mapsto (x\cdot\lambda_s)\delta.
\end{align*}
with $\lambda_s$ as in 
Proposition \ref{monodromy-cone}.
The condition that 
$\cF^\bullet_0$ define a mixed Hodge
structure for the weight filtration 
$W_\bullet$ implies that
we can identify 
\begin{align}\label{ident}T_{[\cF^\bullet_0]}\bD^\vee= 
{\rm Hom}(\cF^2_0,\cF^1_0/\cF^2_0)\simeq {\rm Hom}(\delta^\perp/\delta,\,\bZ\delta)\otimes \bC \simeq \delta^\perp/\delta \otimes \bC\end{align}
where $\bD^\vee=\{x\cdot x=0\}\subset \bP(L_{\rm K3}\otimes \bC)$; the latter spaces are the tangent space
to the mixed period domain for $W_\bullet$. See 
e.g.~\cite[Prop.~2]{carlson}.
Furthermore, we have 
${\rm gr}_\cF^2(N_s)\leftrightarrow 
{\rm gr}^W_2 (N_s)$ under the identification (\ref{ident}).
It follows from (\ref{eq1}, \ref{eq2}) that, 
up to a scalar,
$\partial$ can be identified
with the mapping
$e_s\mapsto \lambda_s\in \delta^\perp/\delta\otimes \bC$.

Hence $r=0$ if and only if
$\lambda_s$ are linearly independent.
More generally, the exact sequence above
identifies the Lie algebra
$H^0(X_0,\Theta_{X_0})$ of
${\rm Aut}^0(X_0)$ with $\ker(\partial)$,
which is the space of linear relations
amongst the monodromy vectors $\lambda_s$.
Since $\rho$ in (\ref{rho})
has finite kernel and toric image,
it follows that
${\rm Aut}^0(X_0)\simeq(\bC^*)^r$.

In the Type II case, we have that
the number $|S|$ 
of connected components of 
$(X_0)_{\rm sing}$ minus one is exactly
the number $r$ of intermediate ruled components
of $X_0$. Furthermore, all $\lambda_s$ are equal
to the same isotropic vector $\lambda$. Hence ${\rm Aut}^0(X_0)\simeq (\bC^*)^r$ where $r=|S|-1$ 
is the rank of the linear
relations amongst the $\lambda_s=\lambda$.
This completes the proof of the first part. \smallskip

We now consider the second part of the proposition.
Let $\bD_\delta\coloneqq 
\{x\in \bD^\vee\,:\,x\cdot \delta\neq 0\}$
and let $U_\delta\subset 
{\rm Stab}_{O(L_{\rm K3})}(\delta)$ 
be the unipotent radical $U_\delta$. 
We have an analytic
open embedding
$$\bD_\delta/U_\delta\hookrightarrow \delta^\perp/\delta\otimes \bC^*\simeq 
(\bC^*)^{20},$$
see \cite[Prop.~4.13]{ae}. Then, the monodromy $\pi_1(B^\circ)\to O(L_{\rm K3})$ factors through
$U_\delta$ and hence, we have a well defined period
map $\Phi^\circ \colon B^\circ \to 
\delta^\perp/\delta
\otimes \bC^*$.
Let $Y(\sigma)$ be the affine
toroidal extension
associated to the rational polyhedral
cone $\sigma = 
\bR_{\geq 0}\{\lambda_s\}_{s\in S}$. By the multivariable nilpotent orbit theorem \cite[Thm.~2.1]{cattani-kaplan}, 
$\Phi^\circ$ extends 
holomorphically to a map 
$\Phi'\colon B\to Y(\sigma)$.

Choose a finite index sublattice 
$L\subset \delta^\perp/\delta$ for which 
$\{\lambda_s\}_{s\in S}$ form part of a basis
of $L$. Then there is a toroidal extension
$X(\sigma)$ of $L\otimes \bC^*$ together 
with a finite map $X(\sigma)\to Y(\sigma)$
restricting to the isogeny $L\otimes \bC^*\to 
\delta^\perp/\delta\otimes \bC^*$ 
on the open
torus orbit. The map $\Phi'$ lifts
to a map $$\Phi\colon B\to X(\sigma).$$

We have that $X(\sigma)\simeq \bC^S
\times (\bC^*)^{20-|S|}$
is smooth, and the restriction
of $\Phi$ to a transverse slice
$\Delta^S\times \{0\}\subset 
\Delta^S\times \Delta^T\simeq B$
transversely slices the deepest
toroidal
boundary
stratum $\partial_\sigma\simeq \{0\}
\times (\bC^*)^{20-|S|}\subset X(\sigma)$.
Let
$
B^{\rm lt}=\{0\}\times\Delta^T
$
and let
$
C^{\rm lt}=C\cap B^{\rm lt}
$
with $C\simeq {\rm smDef}_{X_0}$ 
the slice of Proposition
\ref{reparam-quotient}.
Consider the logarithmic differential
$$d^{\rm log}\Phi_0\colon T_0^{\rm log}B
\to T_{\Phi(0)}^{\rm log}X(\sigma)$$ 
i.e.~the infinitesimal logarithmic 
period map of $\Phi$. 
By (\ref{ks}), its
restriction to the subspace
$T_0 B^{\rm lt}=
H^1(X_0,\Theta_{X_0/B_0}^{\rm log})$
is an isomorphism onto $T_{\Phi(0)}^{\rm log}X(\sigma)$.
Denote this isomorphism by
$\ell$. Let
$
V_\sigma
=\bC\langle\lambda_s\rangle_{s\in S}
$.
We have shown
$
\ell(\partial(e_s))=\lambda_s
$ and
$
\ell\colon {\rm im}(\partial)\to V_\sigma
$ restricts to an isomorphism,
by the hypothesis $r=0$.

The period map $\Phi$ is invariant 
under
the reparameterization action 
of Proposition \ref{reparam-quotient}
and hence factors
through the quotient/slice $C$. 
So we have 
an identification
$$
T_0C^{\rm lt}
\simeq
H^1(X_0,\Theta_{X_0/B_0}^{\rm log})/
{\rm im}(\partial).
$$
We therefore have a commutative diagram
$$
\xymatrix{
0\ar[r]&
{\rm im}(\partial)\ar[r]\ar[d]^{\sim}&
H^1(X_0,\Theta_{X_0/B_0}^{\rm log})
\ar[r]\ar[d]_{\ell}^{\sim}&
T_0C^{\rm lt}\ar[r]\ar[d]^{d\Phi_0}&
0\\
0\ar[r]&
V_\sigma\ar[r]&
T_{\Phi(0)}^{\rm log}X(\sigma)\ar[r]&
T_{\Phi(0)}\partial_\sigma\ar[r]&
0.
}
$$
The bottom row is the logarithmic tangent
sequence of the deepest toric stratum.
The first two vertical arrows are
isomorphisms, and hence so is the third.
So $C^{\rm lt}\to \partial_\sigma$ 
is a local isomorphism. Combined with
the fact that $\Phi$ restricted to $\Delta^S\times \{0\}$ transversely
slices $\partial_\sigma$ we deduce
that $\Phi\vert_C$ itself is a local isomorphism.
\end{proof}

\section{Examples}\label{examples}

Now we arrive at the 
main point of the present
work; to give concrete 
examples of multiparameter
Kulikov models, and explicit recipes for
building them in various cases.

Combinatorial diagrams for snc Type III
K3 surfaces have already appeared quite
a bit in the literature, starting
with \cite{fm}, and with many
more explicit examples in 
\cite{engel1, engel2, engel3, 
engel4, engel5}. It is typical
to draw the intersection complex,
so that the polyhedral
$2$-cells correspond to components of $X_0$,
the $1$-cells to double curves of $X_0$, 
and $0$-cells to point strata of $X_0$.
Thus, in the snc case, all vertices
are trivalent, whereas in the polysnc
case they will be $3$- and $4$-valent.
We first give some snc examples
as a warm-up:

\begin{figure}[h]
    \centering
    \includegraphics[width=\linewidth]{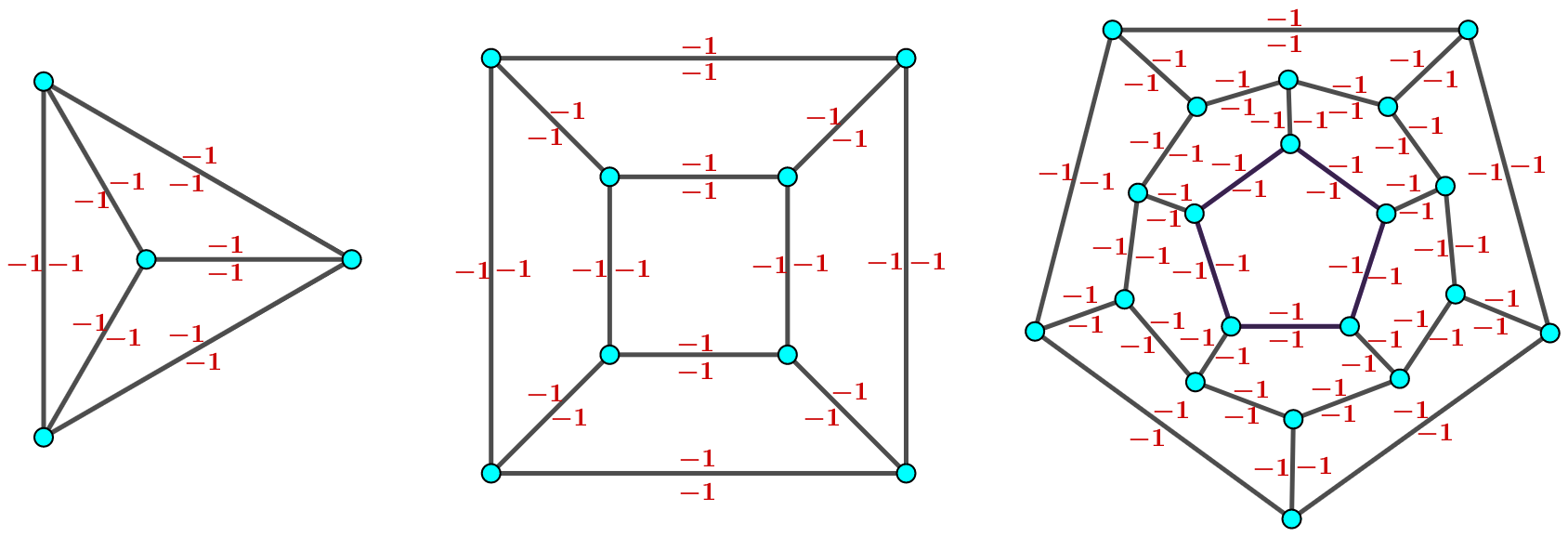}
    \caption{Some snc surfaces built
    from platonic solids}
    \label{fig:platonic}
\end{figure}

\begin{example}[snc platonic solids]
Let $\Gamma$ be a platonic solid
whose faces are triangles. So
$\Gamma \in\{\textrm{tetrahedron, octahedron,
icosahedron}\}$. Then $\Gamma=\Gamma(X_0)$
is the dual complex of an snc surface
$X_0 = \bigcup_i V_i$ defined as follows:
Let $(V^k, D^k)$ be a log CY surface
pair for which $D^k$ is a cycle
of $2\leq k\leq 6$ curves of 
self-intersection $-1$ (the analogue
for $k=1$ is actually a pair $(V^1,D^1)$
where $D^1$ is an irreducible nodal
anticanonical divisor of self intersection
$+1$). 
Then $X_0$ is defined by gluing
copies of $(V^k,D^k)$ exactly
so that $\Gamma(X_0)=\Gamma$. 

For example, $(V^3,D^3)$ is the blow-up
of $(\bP^2, \triangle)$ at six 
smooth points of an anticanonical triangle,
two on each of the 
three lines. $(V^4,D^4)$ is the blow-up
of $(\bP^1\times \bP^1, \square)$
at four points, one on each of the four
edges of an anticanonical square.
And $(V^5,D^5)$ is a del Pezzo surface
of degree $5$ with an anticanonical
pentagon of $(-1)$-curves.
The intersection complexes ${\rm I}(X_0)$ 
of the snc Type III K3 surfaces glued from
these pairs are depicted in Figure
\ref{fig:platonic}. 

Each of the three planar
diagrams in Figure
\ref{fig:platonic} 
should be viewed as living on 
$\bS^2$ with the ``outer'' face 
containing
$\infty\in \bS^2$. The double curves
$D_{ij}=V_i\cap V_j$ are depicted in 
black and the triple points in cyan.
Each edge is
labeled by two red integers; the integer
in the $2$-cell for $V_i$ denotes $D_{ij}^2$
and the integer in the $2$-cell for $V_j$
denotes $D_{ji}^2$.

These surfaces
$X_0$ are examples
of Kulikov models in {\it $(-1)$-form},
since every double curve $D_{ij}$ 
(which is not irreducible
nodal) satisfies $D_{ij}^2=-1$. Up to 
Atiyah flops, every snc Type III model 
can be put in $(-1)$-form 
by the Miranda--Morrison
$(-1)$-theorem \cite[Ch.~V]{prog}.
Deformation classes of 
$(-1)$-forms correspond to convex
triangulations of the sphere \cite{laza}.
\end{example}

\begin{figure}[h]
    \centering
    \includegraphics[width=0.9\linewidth]{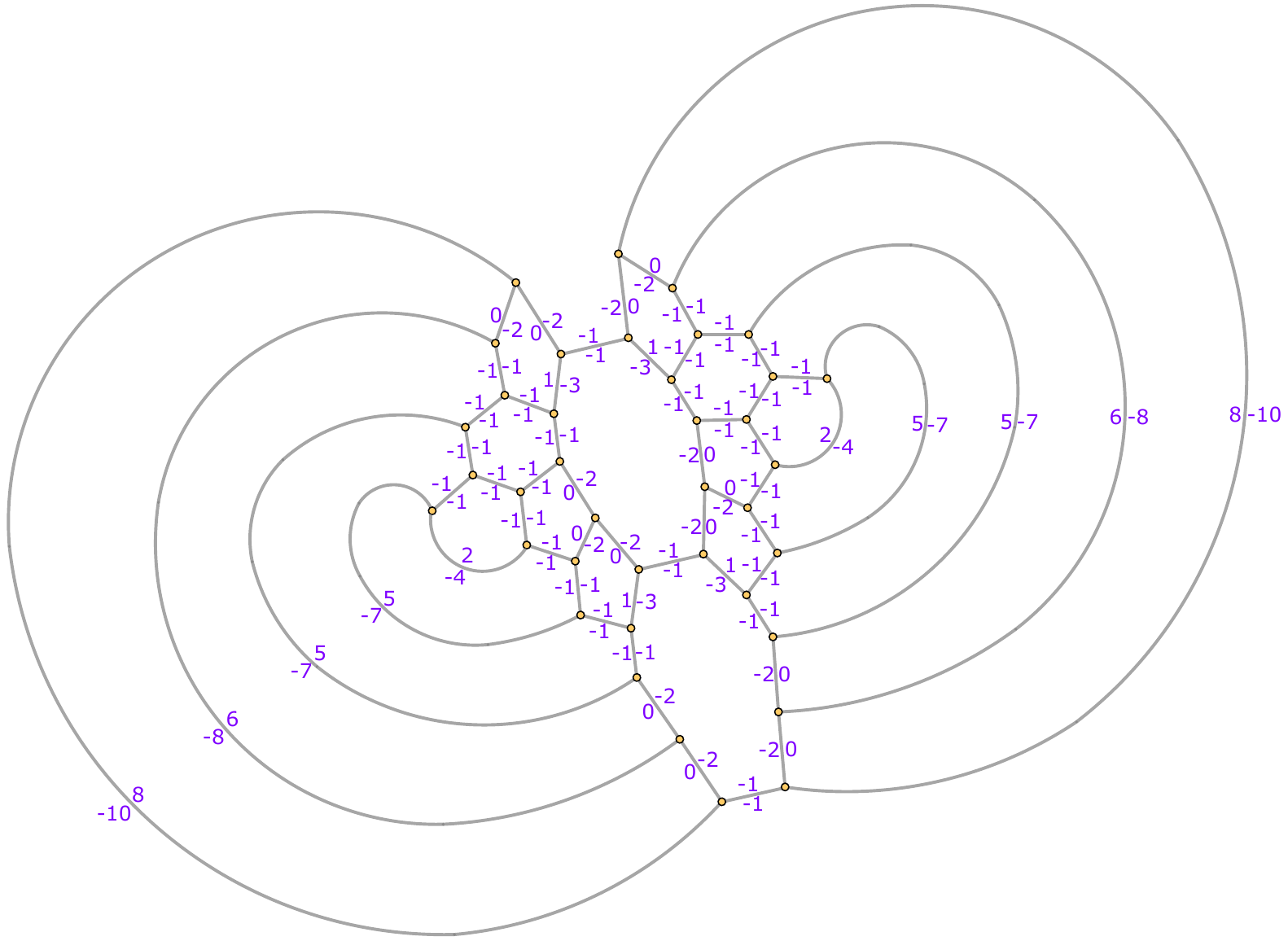}
    \caption{A Kulikov degeneration
    of degree $2$ K3 surfaces from \cite{engel2}. 
    Triple points in yellow, double curves in grey,
    self-intersections $D_{ij}^2$ and $D_{ji}^2$
    in purple.}
    \label{kulikov-ex2}
\end{figure}

Of course, not every Kulikov model
is in $(-1)$-form. For example, see
Figure \ref{kulikov-ex2}. The key
points here are that (1) the depicted
graph is trivalent, (2) the labels
satisfy $D_{ij}^2 + D_{ji}^2=-2$,
and (3) for all cells $i$,
the cycle of integers
$(D_{ij}^2)$ for $j$ cyclically ordered,
appear as the self-intersection
sequence
of an anticanonical cycle on a log CY 
surface pair (a good exercise for the reader
is to check this for a few randomly chosen $2$-cells). \medskip

We now give some examples
of legitimately multiparameter
Kulikov models.

\begin{example}[Rhombicuboctahedron]
We will build a multiparameter Kulikov model
over a $4$-dimensional base $X\to \Delta^4$.
We set $\Gamma$ to be the 
rhombicuboctahedron (one of the
Archimedean solids), which is a
polysimplicial complex of dimension
$2$, consisting
of $n_\triangle = 8$ triangles and 
$n_\square = 18$ squares;
see Figure \ref{fig:rhomb}.
It admits an $S$-coloring as in Definition
\ref{colorable} by the four element set 
$S=\{{\rm red}, \,{\rm yellow},
\,{\rm green}, \, {\rm blue}\}$,
as depicted in Figure \ref{fig:rhomb}.

\begin{figure}
    \centering
    \includegraphics[width=0.4\linewidth]{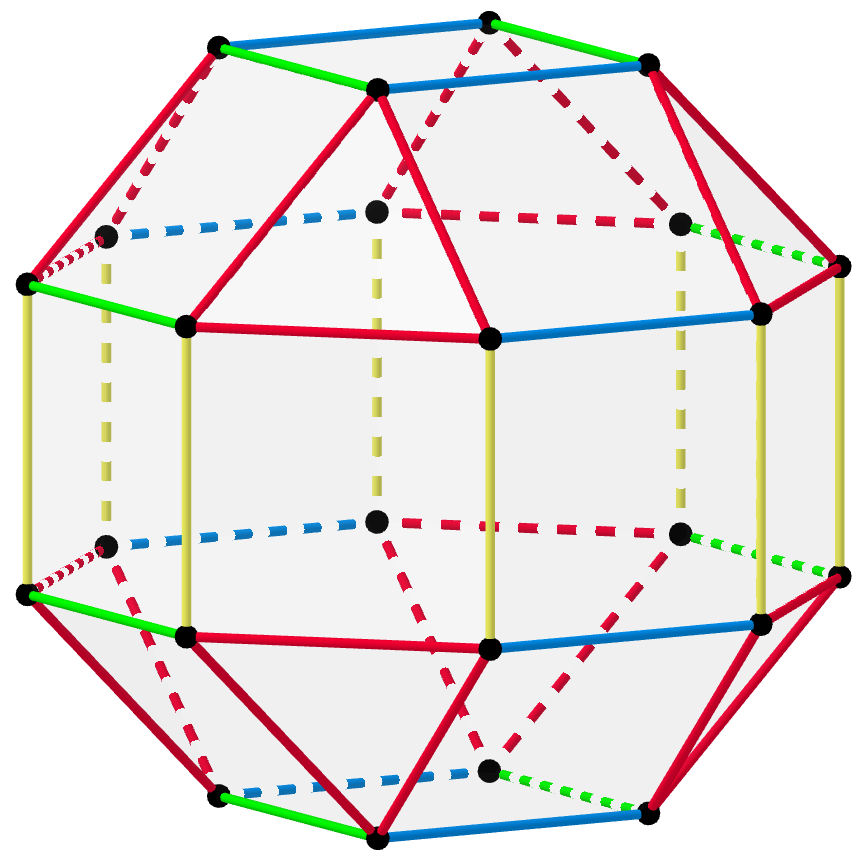}
    \caption{Rhombicuboctahedron, colored
    by four colors}
    \label{fig:rhomb}
\end{figure}

\begin{figure}
    \centering
    \includegraphics[width=0.53\linewidth]{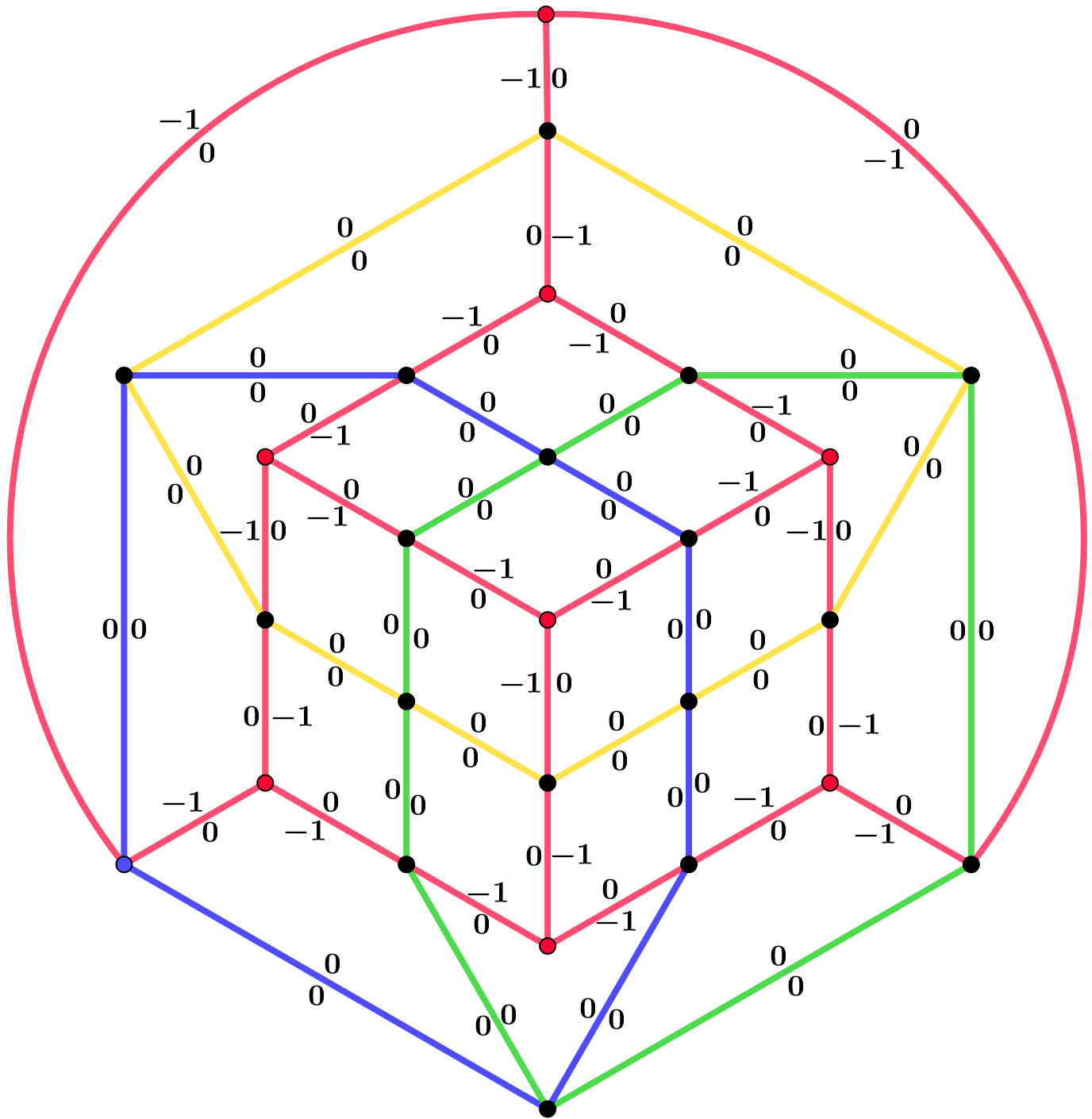}
    \caption{Polysnc Type III K3 
    surface with
    dual complex the rhombicuboctahedron}
    \label{fig:rhomb-dual}
\end{figure}

The corresponding
intersection complex ${\rm I}$ is shown
in the plane in Figure \ref{fig:rhomb-dual}.
We define a polysnc Type III K3
surface
$X_0$ such that $\Gamma(X_0)=\Gamma$
is the rhombicuboctahedron, so
that ${\rm I}(X_0)={\rm I}$. We set 
$X_0=\textstyle \bigcup_{i=1}^{24} (V,D)$
to be a union of $24$ copies of the same
anticanonical pair $(V,D)$, which is the
blow-up of $(\bP^1\times \bP^1,\square)$
at a single non-torus fixed
point on the anticanonical square.
The $24$ copies are glued according
to the combinatorics of Figure
\ref{fig:rhomb-dual}. This is a nice
demonstration of the conservation of charge
formula, Proposition \ref{charge}, since $Q(V,D)=1$.

We now compute the monodromy cone, generated
by $\lambda_{\rm red}$, $\lambda_{\rm yellow}$, 
$\lambda_{\rm green}$, $\lambda_{\rm blue}$
as defined in Proposition \ref{monodromy-cone}.
We have the following values:
\begin{align*}
    &n_{\triangle,\,{\rm red}}=8, &&
    n_{\triangle,\,{\rm yellow}}=
    n_{\triangle,\,{\rm green}}=
    n_{\triangle,\,{\rm blue}}=0,  \\
    &n_{\square, \,{\rm red},\, s} = 4 &&\textrm{for }s\in\{\textrm{yellow, green, blue}\}, \\
    &n_{\square, \,s,\, s'} = 2 &&\textrm{for }s\neq s'\in\{\textrm{yellow, green, blue}\}.
\end{align*}

By Proposition
\ref{monodromy-cone}(3, 4),
the intersection
matrix of 
$(\lambda_{\rm red}, \,\lambda_{\rm yellow},\,
\lambda_{\rm green},\, \lambda_{\rm blue})$
is
$$
\begin{pmatrix}
8 & 4 & 4 & 4 \\  
4 & 0 & 2 & 2 \\  
4 & 2 & 0 & 2 \\  
4 & 2 & 2 & 0 \\  
\end{pmatrix},
$$
which is non-degenerate.
Furthermore, the singular locus
$(X_0)_{\rm sing}^s$ of every
color is connected. 
It follows from Proposition
\ref{monodromy-cone}(5)
that the universal
$d$-semistable deformation
of $X_0$ is $$X\to \Delta^4\times \Delta^{16}$$
where $\Delta^{16}$ parameterizes the locally
trivial, $d$-semistable deformations of $X_0$.
Taking any transverse slice, such
as $\Delta^4\times \{0\}$, we get 
a Kulikov model over a $4$-parameter base.

Over the generic point of $V(u_{\rm red})$, 
the fiber is an snc Type III K3 surface 
in the same locally trivial deformation
class as the octahedral surface depicted
in the middle of Figure \ref{fig:platonic}.
Over the generic point of $V(u_{\rm yellow})$,
$V(u_{\rm green})$, or $V(u_{\rm blue})$,
the fiber is an snc Type II K3 surface,
which is the union of two rational log CY
pairs, along a smooth elliptic curve.
There 
are two blue slabs,
two green slabs, two yellow slabs,
and six red slabs. 
For example, over $V(u_{\rm blue})$,
each blue slab deforms to a smooth
rational surface, and the blue
octagon of $\bP^1$s in $X_0$ deforms
to a smooth elliptic double curve
of the resulting snc Type II K3 surface.
These facts are most easily
ascertained by blurring one's eyes
and focusing on only one color,
in Figure \ref{fig:rhomb-dual}.

Over a $3$-dimensional 
transverse slice of the coordinate axis
$V(u_{\rm yellow},u_{\rm green},u_{\rm blue})$, we get a Kulikov model
with three colors, whose dual
complex is the cube.
\end{example}

\begin{example}[Elliptic K3 surfaces with
an additional involution]

Consider Figure \ref{fig:rainbow}. It defines
a triangle-squarulation $\Gamma$
of $\bS^2$ by gluing
the sequence of $19$ segments along the bottom 
edge, all given by the vector 
$\langle 1,0\rangle$ in $\bR^2$,
to the corresponding
sequence of $19$ segments along
the top edge, successively given by
the vectors $\langle 1, 10-n\rangle$,
for $n=1,\dots,19$.

\begin{figure}
    \centering
    \includegraphics[width=0.8\linewidth]{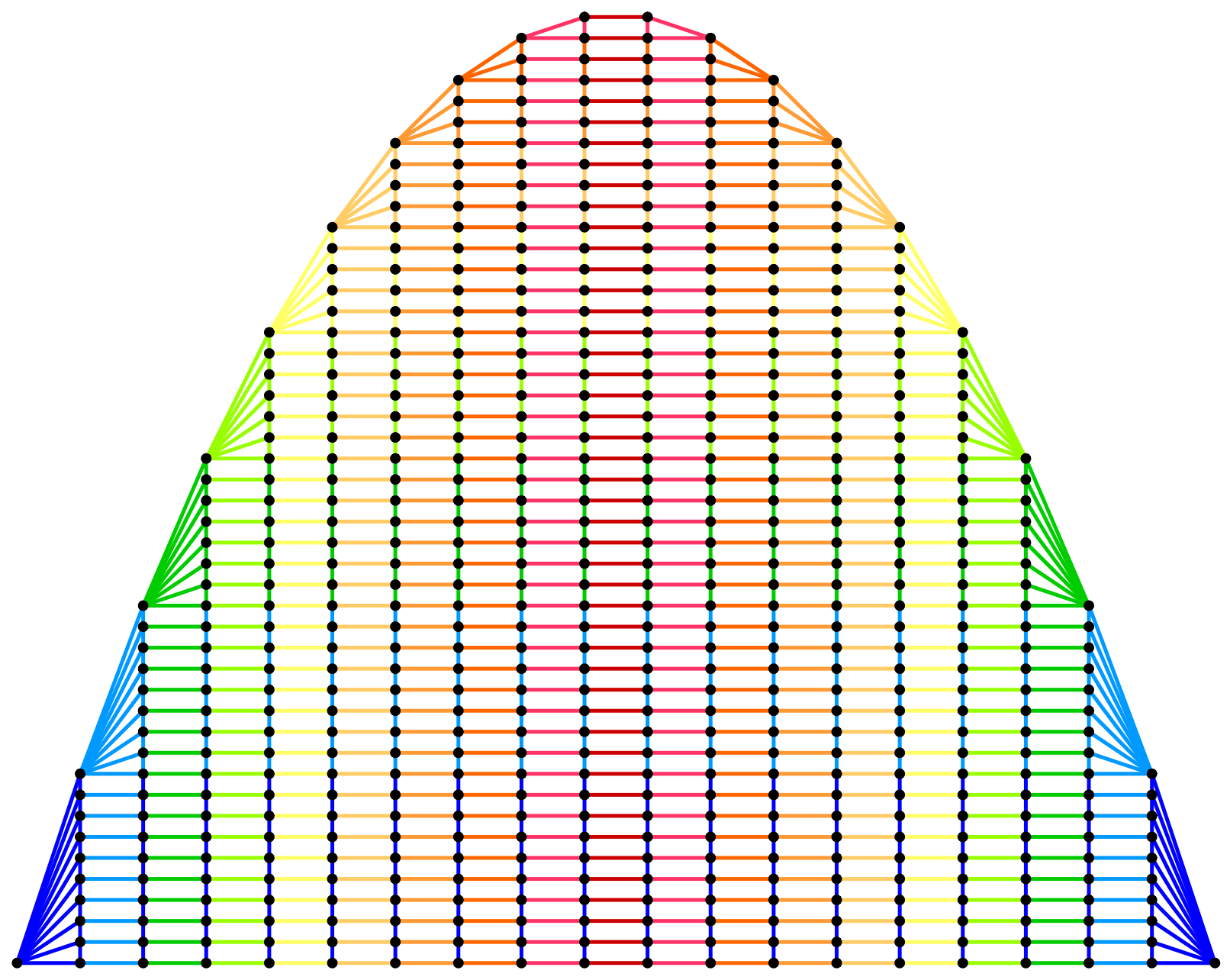}
    \caption{Dual complex
    of a polysnc Type III K3 
    surface, colored by $10$ colors.
Color index:
(1) dark blue,
(2) light blue,
(3) green,
(4) light green,
(5) yellow,
(6) light orange,
(7) orange,
(8) orange-red,
(9) pink-red,
(10) dark red.}
    \label{fig:rainbow}
\end{figure}

The coloring by a $10$-element
set $S=\{1,2,3,4,5,6,7,8,9,10\}$
depicted in Figure \ref{fig:rainbow} 
satisfies Observation \ref{obs}.
Index the colors
successively by the rainbow, 
as in the caption of Figure \ref{fig:rainbow}.
To describe the log CY pairs
appearing as components (corresponding
to the black vertices in 
Figure \ref{fig:rainbow}), we first define
an integral-affine structure on $\bS^2$
by declaring that the gluing map of the
paired edges is the unique shear in 
$\SL_2(\bZ)$ preserving vertical lines
and identifying the given edges.
See \cite[Sec.~7A]{engel3} for further
details. Then, we may interpret the local
integral-affine structure in the neighborhood
of a vertex as the {\it pseudofan}
or {\it tropicalization} of a log
CY pair $(V_i,D_i)$, 
see e.g.~\cite[Sec.~1.2]{ghk2},
\cite[Def.~3.8]{engel1},
\cite[Defs.~5.5, 5.9]{engel3}.

There are $20$ pairs $(V_i,D_i)$
for which $Q(V_i,D_i)>0$ and they 
exactly correspond to the integral-affine
singularities, i.e.~the vertices
along the bottom boundary (note that
this equals the top boundary 
because of the gluing). Their charges
$Q(V_i,D_i)$ are successively
$3,1,1,\dots,1,1,3$,
summing to $24$,
and their deformation types
are denoted $X_3 I_1 I_1\dots I_1 I_1 X_3$
in \cite[Sec.~7B]{engel3}.
For example, $(V_1,D_1)\simeq (V_{20},D_{20})$
are both the unique rational elliptic surface
admitting a cycle of nine $(-2)$-curves as the
anticanonical divisor/double locus.

All other log CY pairs
$(V_j,D_j)$ modeling a vertex are toric;
there are $552$ such toric pairs, corresponding
to the black
lattice points in the strict
interior of Figure \ref{fig:rainbow}.
The vast majority of them are
isomorphic to $(\bP^1\times \bP^1,\square)$.
Then we build a $d$-semistable 
polysnc Type III surface $$\textstyle X_0=\bigcup_{i=1}^{20} (V_i,D_i)\cup \bigcup_{j=21}^{572}(V_j,D_j)$$ by gluing
the log CY pairs along their anticanonical
boundaries, in such a way
that $\Gamma(X_0)=\Gamma$,
and $X_0$ is $d$-semistable, see Proposition
\ref{def-to-dss}.

A refinement $\widetilde{\Gamma}$
decomposing each square into triangles
corresponds to an snc ``resolution''
with
$\Gamma(\wX_0)=\widetilde{\Gamma}$
as in Proposition \ref{homeo-type}.
The resulting snc Type III surfaces
$\wX_0$ which so arise are Kulikov models
of elliptic K3 surface degenerations, as
described by \cite[Sec.~7]{engel3}. 

\begin{table}
\begin{tabular}{|p{0.9em}||c|c|c|c|c|c|c|c|c|c|}
\hline 
\diagbox[width=1.8em, height=1.8em,
innerleftsep=0.1em, innerrightsep=0.3em]{
\raisebox{-0.6ex}[0pt][0pt]{$s'$}
}{\raisebox{0.5ex}[0pt][0pt]{$s$}} 
& $1$ & $2$ & $3$ & $4$ & $5$ & $6$ & $7$ & $8$ & $9$ & $10$  \\
\hline 
\hline
\rule[-0.6em]{0pt}{1.2em} $1$ & $18$ & $18$ & $18$ & $18$ & $18$ & $18$ & $18$ & $18$ & $18$ & $9$ \\
\hline
\rule[-0.6em]{0pt}{1.2em} $2$ & $18$ & $16$ & $16$ & $16$ & $16$ & $16$ & $16$ & $16$ & $16$ & $8$ \\
\hline
\rule[-0.6em]{0pt}{1.2em} $3$ & $18$ & $16$ & $14$ & $14$ & $14$ & $14$ & $14$ &$14$  & $14$ & $7$ \\
\hline
\rule[-0.6em]{0pt}{1.2em} $4$ & $18$ & $16$ & $14$ & $12$ & $12$ & $12$ & $12$ & $12$ & $12$ & $6$ \\
\hline
\rule[-0.6em]{0pt}{1.2em} $5$ & $18$ & $16$ & $14$ & $12$ & $10$ & $10$ & $10$ & $10$ & $10$ & $5$ \\
\hline
\rule[-0.6em]{0pt}{1.2em} $6$ & $18$ & $16$ & $14$ & $12$ & $10$ & $8$ & $8$ & $8$ & $8$ & $4$ \\
\hline
\rule[-0.6em]{0pt}{1.2em} $7$ & $18$ & $16$ & $14$ & $12$ & $10$ & $8$ & $6$ & $6$ & $6$ & $3$ \\
\hline
\rule[-0.6em]{0pt}{1.2em} $8$ & $18$ & $16$ & $14$ & $12$ & $10$ & $8$ & $6$ & $4$ & $4$ & $2$ \\
\hline
\rule[-0.6em]{0pt}{1.2em} $9$ & $18$ & $16$ & $14$ & $12$ & $10$ & $8$ & $6$ & $4$ & $2$ & $1$ \\
\hline
\rule[-0.6em]{0pt}{1.2em}$10$ & $9$ & $8$ & $7$ & $6$ & $5$ & $4$ & $3$ & $2$ & $1$ & $0$ \\
\hline
\end{tabular}
\vspace{10pt}
\caption{Table of values,
with  diagonal entries $n_{\triangle,s}$ 
and off-diagonal entries $n_{\square,s,s'}$}
\label{table}
\end{table}

Regarding counts of triangles
and squares, Table \ref{table} collects
the values. By Proposition \ref{monodromy-cone},
these counts are exactly the Gram matrix
$(\lambda_s\cdot \lambda_{s'})_{s,s'\in S}$.
The matrix is non-degenerate, of determinant $-2^8\cdot 3^2$.
Let $X^+\to \Delta^{S}\times \Delta^{10}$
be the universal $d$-semistable deformation
of $X_0$.
There is a $10$-dimensional nilpotent orbit $\Delta^S\times \{0\}$ corresponding 
exactly to those 
fibers of $X^+$ admitting an $H\oplus E_8(2)$-polarization (we have assumed $X_0$
has trivial periods $\varphi_{X_0}=1$).
Let $$X\to \Delta^S\simeq \Delta^S\times \{0\}$$ 
denote the restriction. 
It is a semistable degeneration
of K3 surfaces, over a $10$-parameter base. Then
$(\Delta^*)^S$ maps  to the $10$-dimensional moduli space 
$F_{H\oplus E_8(2)}$ of lattice-polarized K3 surfaces.
The period image contains an analytic open subset near
the Type III cusp of $F_{H\oplus E_8(2)}$ corresponding
to an isotropic vector $\delta\in (H\oplus E_8(2))^\perp$
satisfying $\delta^\perp/\delta\simeq H\oplus E_8(2)$.
The lattice $\bigoplus_{s\in S} \bZ\lambda_s\subset \delta^\perp/\delta$ generated by the monodromy
cone has index $3$,
for example $\lambda_1$
is $3$-divisible by \cite[Cor.~7.33]{engel3}. This
corresponds to the fact that
the determinant of the Gram matrix
is $-3^2\cdot 2^8$ rather than $-2^8$,
as it would be for a basis of
$H\oplus E_8(2)$.

The fibers $X_t$ of the family $X\to \Delta^S$ 
over $(\Delta^*)^S$ are smooth, elliptic
K3 surfaces $X_t\to \bP^1$ with section, 
admitting an involution $\iota_t$ 
equivariant over an 
involution of $\bP^1$, whose quotient 
$X_t/\iota_t =Y_t\to \bP^1$
is a rational elliptic surface with section. See the
$(r,a,\delta) = (10, 8,0)$ entry of \cite[Sec.~2A]{engel4}.
The fibers $X_t$
over $V(u_{10})\setminus \bigcup_{s\neq 10} V(u_s)$ are snc Type II Kulikov surfaces;
more precisely, the union of two isomorphic
rational elliptic surfaces along the 
same elliptic fiber. All other fibers
of $X\to \Delta^S$ are of Type III.
The number of triple points of the snc
Kulikov models over each coordinate hyperplane
$V(u_s)\setminus \bigcup_{t\neq s}V(u_t)$ are the diagonal entries
of Table \ref{table}.
\end{example}

\begin{remark} It is also possible
to construct polysnc, $d$-semistable
varieties $X_0$ which occus as
central fibers of semistable degenerations
of abelian varieties. These are abstractly smoothable by Theorem \ref{smoothable},
or one can explicitly construct
smoothings using toric geometry
and the Mumford construction; 
see \cite[Constrs.~3.21, 3.38]{survey}.
The combinatorial data required to define
such a degeneration of principally polarized
abelian $g$-folds is a collection
of convex $\bQ{\rm PL}$ functions
$b_s\colon \bR^g\to \bR$, 
for $s\in S$, satisfying
the property that $b_s(x+y)-b_s(x)$
is linear for all $y\in \bZ^g$.
The bending loci 
${\rm Bend}(b_s)\subset \bR^g/\bZ^g$
must satisfy certain
tropical smoothness and transversality 
conditions, to ensure that the resulting
degeneration $X\to \Delta^S$
is semistable. The central
fiber $X_0$ is a union of smooth toric
varieties, whose polytopes are the cells
of the polyhedral decomposition
$\bigcup_{s\in S} {\rm Bend}(b_s)\subset \bR^g/\bZ^g$.
See \cite[Def.~4.18, Ex.~4.20]{survey}
for examples of nodal degenerations 
(see Def.~\ref{semistable}),
which can be associated to a 
regular matroid.
\end{remark}

\section{Acknowledgments}

I thank Nathan Chen and Stefan
Schreieder for useful discussions,
and Simon Felten for his comments
and pointing out the
reference \cite{li}.
This research was partially
supported by NSF
CAREER grant DMS-2441240
and a Sloan Research Fellowship.

ChatGPT was used to critically review an 
earlier version of this manuscript. It helped identify a number of errors,
and suggested material for the revised draft (Corollary \ref{universal}, Propositions \ref{reparam-action}, \ref{reparam-quotient}, \ref{bigger}, and Lemma \ref{li}). Mathematical arguments and conclusions were independently verified by the author, who takes full responsibility for the contents of the paper.

\bibliographystyle{plain}
\bibliography{bibliography}

\end{document}